\documentclass{amsart}[10pt]
\usepackage{amssymb,amsmath,amsthm,graphics,amscd,amsfonts}
\usepackage{color}

\theoremstyle{plain}
\newtheorem{theorem}{Theorem}[section]
\newtheorem{proposition}[theorem]{Proposition}

\newtheorem{lemma}[theorem]{Lemma}

\newtheorem{problem}[theorem]{Problem}

\newtheorem{definition-proposition}[theorem]{Definition-Proposition}

\newcounter{xpl}[section]

\newenvironment{remark}{\refstepcounter{theorem} \medskip \noindent {\bf  Remark \arabic{section}.\arabic{theorem}}}{\hfill\mbox{}\bigskip}

\newcounter{thmlist}

\newcommand{\bfC}{{\mathbb C}}

\newcommand{\barpartial}{{\overline \partial}}

\newcommand{\mapright}[1]{\smash{\mathop{   \hbox to 0.7cm{\rightarrowfill}}
 \limits^{#1}}}

\def\w{\textbf{w}}

\begin{document}

\title{An example of asymptotically Chow unstable manifolds with constant scalar curvature}

\author{Hajime Ono}
\address{Department of Mathematics,
Faculty of Science and Technology,
Tokyo University of Science,
2641 Yamazaki, Noda,
Chiba 278-8510, Japan}
\email{ono\_hajime@ma.noda.tus.ac.jp}

\author{Yuji Sano}
\address{Kumamoto University\\
Graduate School of Science and Technology\\
2-39-1, Kurokami\\
Kumamoto, 860-8555 Japan}
\email{sano@sci.kumamoto-u.ac.jp}

\author{Naoto  Yotsutani}
\address{University of Science and Technology of China\\
School of Mathematics\\
Hefei, Anhui\\
230026 P.R. China}
\email{naotoy@ustc.edu.cn}
    
\date{\today}

\pagestyle{plain}


\begin{abstract} 
Donaldson proved that if a polarized manifold $(V, L)$ has constant
  scalar curvature K\"ahler metrics in $c_1(\!L)$ and its automorphism
  group $\mathrm{Aut}(V\!\!,\!L)$ is discrete, $(V,L)$ is asymptotically
  Chow stable.  In this paper, we shall show an example which implies
  that the above result does not hold in the case where
  $\mathrm{Aut}(V, L)$ is not discrete.

\end{abstract}
\keywords{asymptotic Chow stability,  K\"ahler metric of constant scalar curvature, toric Fano manifold, Futaki invariant}¡¡
\subjclass{Primary 53C55, Secondary 53C21, 55N91}

\maketitle

\section{Introduction}
One of the main issues in K\"ahler geometry is the existence problem
of K\"ahler metrics with constant scalar curvature on a given K\"ahler
manifold.  Through Yau's conjecture \cite{yau0602} and the works of
Tian \cite{tian97}, Donaldson~\cite{donaldson02}, this problem is
formulated as follows; {\sl The existence of K\"ahler metrics with
  constant scalar curvature in a fixed integral K\"ahler class would
  be equivalent to a suitable notion of stability of manifolds in the
  sense of Geometric Invariant Theory.}  Though remarkable progress is
made recently in this problem, we shall focus only on the related
results to our purpose.  Let $(V, L)$ be an $m$-dimensional polarized
manifold, that is to say, $L\to V$ is an ample line bundle over an
$m$-dimensional compact complex manifold $V$.  Then the first Chern
class $c_1(L)$ of $L$ can be regarded as a K\"ahler class of $V$.  Let
$\mathrm{Aut}(V, L)$ be the group of holomorphic automorphisms of $(V,
L)$ modulo the trivial automorphism
$\mathbb{C}^\times:=\mathbb{C}-\{0\}$.  In \cite{donaldson01},
Donaldson proved that

\begin{theorem}[Donaldson]\label{thm:donaldson}
  Let $(V,L)$ be a polarized manifold.  Assume that $\mathrm{Aut}(V,
  L)$ is discrete.  If $(V,L)$ has constant scalar curvature K\"ahler
  (cscK) metrics in $c_1(L)$, $(V, L)$ is asymptotically Chow stable.
\end{theorem}

The purpose of this paper is to show an example of asymptotically Chow
unstable polarized manifolds with cscK metrics in the case where
$\mathrm{Aut}(V,L)$ is not discrete.  To state our result more
precisely, let us recall the definition of asymptotic Chow stability
and some related results.  Since $L$ is ample, $V$ can be embedded
into the projective space $\mathbb{P}(W):=\mathbb{P}(H^0(V,L^k)^*)$
for sufficiently large $k$ as an algebraic variety $\Psi_{L^k}(V)$.
For $\Psi_{L^k}(V)$, there corresponds to a point
$[\mathrm{Ch}(\Psi_{L^k}(V))]$ in
$\mathbb{P}[\mathrm{Sym}^d(W)^{\otimes(m+1)}]$, which is often called
the Chow point ({\textit{cf.}} see \cite{mum-fog-kir01} for the full detail).
Take an element $\mathrm{Ch}(\Psi_{L^k}(V))$ representing the Chow
point $[\mathrm{Ch}(\Psi_{L^k}(V))]$.  The action of the special
linear group $\mathrm{SL}(W,\mathbb{C})$ on $W$ is extended to the
action on $\mathrm{Sym}^d(W)^{\otimes(m+1)}$.  We call $\Psi_{L^k}(V)$
Chow stable if and only if the orbit
$\mathrm{SL}(W,\mathbb{C})\cdot\mathrm{Ch}(\Psi_{L^k}(V))$ is closed
and its stabilizer is finite.  We call it Chow semistable if and only
if the closure of the orbit does not contain the origin.  Also we call
$(V,L)$ asymptotically Chow (semi-)stable if and only if
$\Psi_{L^k}(V)$ is Chow (semi-)stable for all sufficiently large $k$.
In this paper, we say that $(V, L)$ is asymptotically Chow unstable if
$(V,L)$ is not asymptotically Chow semistable.  
Theorem~\ref{thm:donaldson} is extended by Mabuchi \cite{mabuchi05} to the
case where $\mathrm{Aut}(V, L)$ is not discrete.

\begin{theorem}[Mabuchi]\label{thm:mabuchi}
  Let $(V, L)$ be a polarized manifold.  Assume that the obstruction
  introduced in \cite{mabuchi04} vanishes.  If $(V,L)$ has cscK
  metrics in $c_1(L)$, $(V, L)$ is asymptotically Chow polystable in
  the sense of \cite{mabuchi05}.
\end{theorem}

The notion of polystability in the above is defined by that the orbit
of $\Psi_{L^k}(V)$ with respect to the action of $\mathrm{SL}(W,
\mathbb{C})$ is closed.  So polystability implies semistability.  The
obstruction in the above is defined in \cite{mabuchi04} as a necessary
condition for $(V, L)$ to be asymptotically Chow semistable.  This
obstruction is reformulated by Futaki \cite{futaki04} in more general
form by generalizing so-called Futaki invariant.  The original Futaki
invariant \cite{futaki83} is a map $f:\mathfrak{h}(V)\to \mathbb{C}$
defined by
\[
 f(X):=\int_V Xh_\omega \omega^m,
\]
where $\mathfrak{h}(V)$ is the Lie algebra of all holomorphic vector
fields on $V$, $\omega$ is a K\"ahler form and $h_\omega$ is a
real-valued function defined by
\[
s(\omega)-\biggl(\int_V s(\omega) \omega^m \biggl/ \int_V \omega^m \biggr)
	=-\Delta_\omega h_\omega
\]
up to addition of a constant.  Here $s(\omega)$ denotes the scalar
curvature of $\omega$, $(g^{i\bar{j}})_{i\bar{j}}$ denotes the inverse
of $(g_{i\bar{j}})_{i\bar{j}}$, and
$\Delta_\omega:=-g^{i\bar{j}}\partial_i\bar{\partial}_j$ denotes the
complex Laplacian with respect to $\omega$.  It is well-known that $f$
is independent of the choice of $\omega$ and that the vanishing of $f$
is an obstruction for the existence of cscK metrics in the K\"ahler
class $[\omega]$.

Now let us recall the definition of Futaki's obstruction for
asymptotic Chow semistability.  Let $\mathfrak{h}_0(V)$ be the Lie
subalgebra of $\mathfrak{h}(V)$ consisting of holomorphic vector
fields which have non-empty zero set.  For any $X \in {\mathfrak
  h}_0(V)$, there exists a complex valued smooth function $u_X$ such
that
\[
i(X) \omega = - \barpartial u_X,
\]
\begin{equation}\label{eq:normalization}
\int_V u_X\,\omega^m = 0.
\end{equation}
Let $\theta$ be a type $(1,0)$ connection of the holomorphic tangent
bundle $T'V$.  Let $\Theta := \bar\partial\theta$, which is the
curvature form with respect to $\theta$.  For $X\in \mathfrak{h}(V)$,
let $ L(X) := \nabla_X - \mathcal{L}_X$, where $\nabla_X$ and
$\mathcal{L}_X$ are the covariant derivative by $X$ with respect to
$\theta$ and the Lie derivative respectively.  Remark that $L(X)$ can
be regarded as a smooth section of $\mathrm{End}(T'V)$ the
endomorphism bundle of the holomorphic tangent bundle.  Let $\phi$ be
a $GL(m,\bfC)$-invariant polynomial of degree $p$ on ${\mathfrak
  gl}(m, \bfC)$.  We define ${\mathcal F}_{\phi} : \mathfrak h_0(V)
\to \bfC$ by
\begin{align}\label{eq:generalized_futaki_inv}
{\mathcal F}_{\phi}(X) 
= & \, (m-p+1) \int_V \phi(\Theta) \wedge u_X\,\omega^{m-p}
\\ 
\nonumber
& 
+ \int_V \phi(L(X) + \Theta) \wedge \omega^{m-p+1}.
\end{align}
It is proved that ${\mathcal F}_{\phi}(X)$ is independent of the
choices of $\omega$ and $\theta$ (see \cite{futaki04}).  Let
$\mathrm{Td}^p$ be the $p$-th Todd polynomial which is a
$GL(m,\bfC)$-invariant polynomial of degree $p$ on ${\mathfrak gl}(m,
\bfC)$.  Then it is proved \cite{futaki04}

\begin{theorem}[Futaki]\label{thm:futaki04}
  If $(V, L)$ is asymptotically Chow semistable, then, for any
  $p=1,\cdots, m$, $\mathcal{F}_{\mathrm{Td}^p}(X)=0$ for $X$ in a
  maximal reductive subalgebra $\mathfrak{h}_r(V)$ of
  $\mathfrak{h}_0(V)$.
\end{theorem}

In particular $\mathcal{F}_{\mathrm{Td}^1}$ is equal to
$f\mid_{\mathfrak{h}_0(V)}$ up to multiplication of a constant.  The
vanishing of $\mathcal{F}_{\mathrm{Td}^p}$ for all $p$ is equivalent
to the vanishing of Mabuchi's obstruction ({\textit{cf.}} Proposition 4.1 in
\cite{futaki04}).

\begin{remark}
  It might be noticed among experts that the main result in
  \cite{futaki04} derives a stronger statement than 
Theorem~\ref{thm:futaki04}.  It says that the vanishing of
  $\mathcal{F}_{\mathrm{Td}^p}$ for all $p$ follows from Chow
  semistability of $(V, L^{k_i})$ for some sequence $\{k_i\}_i$ of
  integers (not necessarily asymptotic Chow semistability).  In fact,
  it is proved in \cite{futaki04} that Chow semistability of $(V,L^k)$
  implies the equation (4.2) in \cite{futaki04} for a given $k$.  The
  invariants $\mathcal{F}_{\mathrm{Td}^p}$ correspond to the
  coefficients of the polynomial in $k$ of degree $m+1$ in the right
  hand of (4.2) in \cite{futaki04}.  Hence, the vanishing of the
  coefficients is implied by the vanishing of the polynomial not
  necessarily for all $k$ greater than some positive integer $k_0$ but
  just for finitely many $k$.  Related to this remark, a necessary
  condition for Chow semistability of polarized toric manifolds is
  studied by the first author \cite{ono_JMSJ}.
\end{remark}

In \cite{futaki-ono-sano0811}, Futaki and the first and second authors
investigated the linear dependence among
$\{\mathcal{F}_{\mathrm{Td}^p}\}_p$ and proposed the following
question.

\begin{problem}\label{problem:futaki-ono-sano}
  Does the existence of cscK metrics induce the vanishing of
  $\mathcal{F}_{\mathrm{Td}^p}$ for all $p$?
\end{problem}

For $p=1$, the existence of cscK metrics of course implies the
vanishing of $\mathcal{F}_{\mathrm{Td}^1}$.  If the answer were
affirmative, the assumption of Theorem~\ref{thm:mabuchi} could be
omitted and Theorem~\ref{thm:donaldson} could be extended to the case
where $\mathrm{Aut}(V,L)$ is not discrete.  Note that this extension
is also discussed in Conjecture~1 in \cite{arezzo-loi04}.

Moreover, it was claimed in \cite{futaki-ono-sano0811} that if a
counterexample to Problem~\ref{problem:futaki-ono-sano} exists among
toric Fano manifolds with anticanonical polarization, it should be a
non-symmetric toric Fano manifold with K\"ahler-Einstein metrics in
the sense of Batyrev-Selivanova \cite{batyrev-selivanova99}.  At the
time when \cite{futaki-ono-sano0811} was written, the existence of
such toric Fano manifolds was not known.  However it is discovered by
Nill-Paffenholz \cite{nill-paffenholz090505} very recently.  The main
result of this paper is to show that one of the toric Fano manifolds
in \cite{nill-paffenholz090505} is the desired example in
\cite{futaki-ono-sano0811}.  That is to say,

\begin{theorem}\label{thm:main}
  There exists a seven dimensional toric Fano manifold $V$ with
  K\"ahler-Einstein metrics in $c_1(V):=c_1(K^{-1}_V)$, whose
  $\mathcal{F}_{\mathrm{Td}^p}$ does not vanish for $2\le p \le 7$.
  In particular, $(V,K_V^{-k})$ is not Chow semistable for all $k$
  greater than some positive integer $k_0$.
\end{theorem}

Also Theorem \ref{thm:main} implies that the assumption about
obstruction in Theorem~\ref{thm:mabuchi} can not be omitted.  Hence,
this means that our example in Theorem~\ref{thm:main} is also a
counterexample to Conjecture~1 in \cite{arezzo-loi04}.

We shall prove Theorem \ref{thm:main} by the following two ways; the
direct calculation by the localization formula 
(Section~\ref{sec:localization}), and the derivation of the Hilbert series
(Section~\ref{sec:hilbert}).  In particular, our method implies the
existence of K\"ahler-Einstein metrics on $V$ independently of
\cite{nill-paffenholz090505}.  Remark that on Fano manifolds, all cscK
metrics in $c_1(V)$ are equal to K\"ahler-Einstein metrics.

\section{The Nill-Paffenholz's example}\label{sec:NP}

See \cite{nill-paffenholz090505} for notations and terminologies of
toric geometry in this section.

First of all, let us recall toric Fano manifolds briefly.  A toric
variety $V$ is an algebraic normal variety with an effective
holomorphic action of $T_{\mathbb{C}}:=(\mathbb{C}^\times)^m$, where
$\dim_\mathbb{C}V=m$.  Let $T_\mathbb{R}:=(S^1)^m$ be the real torus
in $T_{\mathbb{C}}$ and $\mathfrak{t}_{\mathbb{R}}$ be the associated
Lie algebra.  Let $N_\mathbb{R}:=J\mathfrak{t}_\mathbb{R}\simeq
\mathbb{R}^m$ where $J$ is the complex structure of
$T_{\mathbb{C}}$. Let $M_\mathbb{R}$ be the dual space
$Hom(N_\mathbb{R}, \mathbb{R})\simeq\mathbb{R}^m$ of
$N_\mathbb{R}$. Denoting the group of algebraic characters of
$T_{\mathbb{C}}$ by $M$, then $M_\mathbb{R}=M\otimes_\mathbb{Z}
\mathbb{R}$.  It is well-known that $m$-dimensional compact toric
manifolds correspond to nonsingular complete fans in
$\mathbb{R}^m$. Moreover when $V$ is an $m$-dimensional toric Fano
manifold, the corresponding fan $\Sigma_V\subset N_{\mathbb{R}}\simeq
\mathbb{R}^m$ satisfies the following properties: Let $N\subset
N_{\mathbb{R}}$ be the dual lattice of $M$,
$$G_V=\{\sigma\in N \, |\, \mathbb{R}_{>0}\cdot \sigma \in \Sigma_V
\text{ and }\sigma \text{ is primitive}\}$$
and $Q_V$ be the convex hull of $G_V$ in $\mathbb{R}^m$. Then
\begin{enumerate}
\item[(a)] the set of vertices of $Q_V$ is equal  to $G_V$,
\item[(b)] the origin of $N_{\mathbb{R}}$ is contained in the
  interior of $Q_V$,
\item[(c)] any face of $Q_V$ is a simplex, and
\item[(d)] the set of vertices of any facet of $Q_V$ constitutes a basis
  of $N\simeq \mathbb{Z}^m\subset N_{\mathbb{R}}$.
\end{enumerate}
An integral polytope satisfying the conditions (b), (c) and (d) is
often called a Fano polytope. Conversely, if an $m$-dimensional Fano
polytope $Q\subset \mathbb{R}^m$ is given, then
$$\Sigma(Q):=\{0\}\cup \{c(F)\}_{F\text{: face of }Q}$$
is a nonsingular complete fan in $\mathbb{R}^m$. Here $c(F)=\mathbb{R}_{\ge 0}
\cdot F\subset \mathbb{R}^m$ is the cone over $F$. Hence there is the
$m$-dimensional toric Fano manifold $V$ associated with the fan $\Sigma(Q)$.
By the construction above, $Q_V=Q$.

Let $V$ be the seven dimensional toric Fano manifold whose vertices of
Fano polytope $Q_V$ in $N_\mathbb{R}\simeq \mathbb{R}^7$ are given by
\begin{align}
\nonumber
&
\\
\nonumber
&
(\mathbf{v}_1 \,\, \mathbf{v}_2 \,\, \mathbf{v}_3 \,\, 
\mathbf{v}_4 \,\, \mathbf{v}_5 \,\, \mathbf{v}_6 \,\, 
\mathbf{v}_7 \,\, \mathbf{v}_8 \,\, \mathbf{v}_9 \,\, 
\mathbf{v}_{10} \,\, \mathbf{v}_{11} \,\, \mathbf{v}_{12})
\\
\label{eq:vertices_fano_polytope}
&=
\left(\begin{array}{rrrrrrrrrrrr}
1 & 0 & 0 & 0 & 0 & -1 & 0 & 0 & 0 & 0 & 0 & 0
\\
0 & 1 & 0 & 0 & -1 & 0 & 0 & 0 & 0 & 0 & 0 & 0
\\
0 & 0 & 1 & -1 & 0 & 0 & 0 & 0 & 0 & 0 & 0 & 0
\\
0 & 0 & 0 & 0 & 0 & 0 & 1 & 0 & 0 & -1 & 0 & 0
\\
0 & 0 & 0 & 0 & 0 & 0 & 0 & 1 & 0 & -1 & 0 & 0
\\
0 & 0 & 0 & 0 & 0 & 0 & 0 & 0 & 1 & -1 & 0 & 0
\\
0 & 0 & 0 & -1 & -1 & -1 & 0 & 0 & 0 & 2 & 1 & -1
\end{array}\right).
\end{align}
Remark that $V$ is isomorphic to a $\mathbb{P}^1$-bundle over
$(\mathbb{P}^1)^3\times \mathbb{P}^3$.  To see this, let
$\Tilde{Q}\subset \mathbb{R}^6$ be the Fano polytope whose vertices
are
\begin{align*}
(\Tilde{\mathbf{v}}_1 \,\, &\Tilde{\mathbf{v}}_2 \,\, \Tilde{\mathbf{v}}_3 
\,\, \Tilde{\mathbf{v}}_4 \,\, \Tilde{\mathbf{v}}_5 \,\, \Tilde{\mathbf{v}}_6
\,\, \Tilde{\mathbf{v}}_7 \,\, \Tilde{\mathbf{v}}_8 \,\, \Tilde{\mathbf{v}}_9
\,\, \Tilde{\mathbf{v}}_{10})\\
&=
\left(\begin{array}{rrrrrrrrrr}
1 & 0 & 0 & 0 & 0 & -1 & 0 & 0 & 0 & 0
\\
0 & 1 & 0 & 0 & -1 & 0 & 0 & 0 & 0 & 0
\\
0 & 0 & 1 & -1 & 0 & 0 & 0 & 0 & 0 & 0
\\
0 & 0 & 0 & 0 & 0 & 0 & 1 & 0 & 0 & -1
\\
0 & 0 & 0 & 0 & 0 & 0 & 0 & 1 & 0 & -1
\\
0 & 0 & 0 & 0 & 0 & 0 & 0 & 0 & 1 & -1
\end{array}\right).
\end{align*}
It is easy to see that the $6$-dimensional toric Fano manifold
associated to $\Tilde{Q}$ is $(\mathbb{P}^1)^3\times \mathbb{P}^3$.
The projection $\pi:\mathbb{Z}^7\to \mathbb{Z}^6,\ \pi(x_1,\dots,x_7)
=(x_1,\dots,x_6)$ is a map of fans from $(\mathbb{Z}^7,\Sigma(Q_V))$
to $(\mathbb{Z}^6,\Sigma(\Tilde{Q}))$.  Hence we have an equivariant
morphism $p:V\to (\mathbb{P}^1)^3\times \mathbb{P}^3$ associated to
$\pi$. We can apply Proposition $1.33$ of \cite{o} to the map of fans
$\pi$. As a result, $p$ is a $\mathbb{P}^1$- fibration on
$(\mathbb{P}^1)^3\times \mathbb{P}^3$.

\begin{theorem}[Nill-Paffenholz]\label{thm:Nill-Paffenholz}
  The toric Fano manifold $V$ defined by
  (\ref{eq:vertices_fano_polytope}) is not symmetric, but its Futaki
  invariant vanishes.  In particular $V$ admits
  $T_\mathbb{R}$-invariant K\"ahler-Einstein metrics.
\end{theorem}

The second statement in Theorem \ref{thm:Nill-Paffenholz} follows from
the fact proved by Wang-Zhu \cite{wang-zhu01}, which says that a toric
Fano manifold admits K\"ahler-Einstein metrics if and only if its
Futaki invariant vanishes.  Here we shall explain about the symmetry
of toric Fano manifolds in Theorem~\ref{thm:Nill-Paffenholz}.  Let
$\mbox{Aut}(V)$ be the group of automorphisms of $V$.  Let
$\mathcal{W}(V)$ be the Weyl group of $\mbox{Aut}(V)$ with respect to
the maximal torus and $N_\mathbb{R}^{\mathcal{W}(V)}$ be the
$\mathcal{W}(V)$-invariant subspace of $N_\mathbb{R}$.  Batyrev and
Selivanova \cite{batyrev-selivanova99} say that a toric Fano manifold
$V$ is symmetric if and only if $\dim
N_\mathbb{R}^{\mathcal{W}(V)}=0$.

Then, let us consider the symmetry of $V$ defined by
(\ref{eq:vertices_fano_polytope}).  $\mathcal{W}(V)$ contains two
cyclic groups acting on $(\mathbb{P}^1)^3$ and $\mathbb{P}^3$
respectively, \textit{i.e.}, one acts on $(x_1, x_2, x_3)$ and the other acts
on $(x_4, x_5, x_6)$ where $(x_1,x_2,x_3,x_4,x_5,x_6,x_7)$ $\in
N_{\mathbb{R}} \simeq \mathbb{R}^7$.  Hence, we find that the
dimension of $N_\mathbb{R}^{\mathcal{W}(V)}$ is at most one.  However,
since $V$ is not symmetric, $\dim N_\mathbb{R}^{\mathcal{W}(V)}=1$.

Next, we shall consider affine toric varieties in $V$ and the
associated $7$-dimensional cones.  As explained above, we find that in
(\ref{eq:vertices_fano_polytope}), the first six vertices
$\{\mathbf{v}_1, \ldots, \mathbf{v}_6\}$ give affine toric varieties
in $(\mathbb{P}^1)^3$, the next four vertices $\{\mathbf{v}_7, \ldots,
\mathbf{v}_{10}\}$ give them in $\mathbb{P}^3$, and the last two
vertices $\{\mathbf{v}_{11}, \mathbf{v}_{12}\}$ give them in the
$\mathbb{P}^1$-fibre.  More precisely, the set of vertices of each
facet of the Fano polytope defined by
(\ref{eq:vertices_fano_polytope}) consists one of $\{\mathbf{v}_1,
\mathbf{v}_6\}$, one of $\{\mathbf{v}_2, \mathbf{v}_5\}$, one of
$\{\mathbf{v}_3, \mathbf{v}_4\}$, three of $\{\mathbf{v}_7,\ldots,
\mathbf{v}_{10}\}$ and one of $\{\mathbf{v}_{11}, \mathbf{v}_{12}\}$.
Hence, the toric Fano manifold $V$ is covered by $64$ affine toric
varieties, which are isomorphic to $\mathbb{C}^7$ as listed 
in Table~\ref{tab:coordinate}.
\begin{table}
\begin{center}
\begin{tabular}{|c|l|}
  \hline
cone & toric affine variety $\simeq \mathbb{C}^7$
\\
\hline
$\{\mathbf{ v}_1,\mathbf{ v}_2, \mathbf{ v}_3, \mathbf{ v}_7, 
	\mathbf{ v}_8, \mathbf{ v}_9, \mathbf{ v}_{11}\}$
&
$\mbox{Spec}(\mathbb{C}[X_1,X_2,X_3,Y_1,Y_2,Y_3,Z])$
\\
\hline
$\{\mathbf{ v}_6,\mathbf{ v}_2, \mathbf{ v}_3, \mathbf{ v}_7, 
\mathbf{ v}_8, \mathbf{ v}_9, \mathbf{ v}_{11}\}$
&
$\mbox{Spec}(\mathbb{C}[X_1^{-1},X_2,X_3,Y_1,Y_2,Y_3,ZX_1^{-1}])$
\\
\hline
$\{\mathbf{ v}_1,\mathbf{ v}_2, \mathbf{ v}_3, \mathbf{ v}_7, 
\mathbf{ v}_8, \mathbf{ v}_{10}, \mathbf{ v}_{11}\}$
&
$\mbox{Spec}(\mathbb{C}[X_1,X_2,X_3,
Y_1Y_3^{-1},Y_2Y_3^{-1},Y_3^{-1},ZY_3^{2}])$
\\
\hline
$\{\mathbf{ v}_1,\mathbf{ v}_2, \mathbf{ v}_3, \mathbf{ v}_7, 
\mathbf{ v}_8, \mathbf{ v}_9, \mathbf{ v}_{12}\}$
&
$\mbox{Spec}(\mathbb{C}[X_1,X_2,X_3,Y_1,Y_2,Y_3,Z^{-1}])$
\\
\hline
$\{\mathbf{ v}_6,\mathbf{ v}_5, \mathbf{ v}_3, \mathbf{ v}_7, 
\mathbf{ v}_8, \mathbf{ v}_9, \mathbf{ v}_{11}\}$
&
$\mbox{Spec}(\mathbb{C}[X_1^{-1},X_2^{-1},X_3,Y_1,Y_2,Y_3,
ZX_1^{-1}X_2^{-1}])$
\\
\hline
$\{\mathbf{ v}_6,\mathbf{ v}_2, \mathbf{ v}_3, \mathbf{ v}_8, 
\mathbf{ v}_9, \mathbf{ v}_{10}, \mathbf{ v}_{11}\}$
&
$\mbox{Spec}(\mathbb{C}[X_1^{-1},X_2,X_3,
Y_1^{-1},Y_2Y_1^{-1},Y_3Y_1^{-1},ZX_1^{-1}Y_1^2])$
\\
\hline
$\{\mathbf{ v}_6,\mathbf{ v}_2, \mathbf{ v}_3, \mathbf{ v}_7, 
\mathbf{ v}_8, \mathbf{ v}_9, \mathbf{ v}_{12}\}$
&
$\mbox{Spec}(\mathbb{C}[X_1^{-1},X_2,X_3,Y_1,Y_2,Y_3,Z^{-1}X_1])$
\\
\hline
$\{\mathbf{ v}_1,\mathbf{ v}_2, \mathbf{ v}_3, \mathbf{ v}_8, 
\mathbf{ v}_9, \mathbf{ v}_{10}, \mathbf{ v}_{12}\}$
&
$\mbox{Spec}(\mathbb{C}[X_1,X_2,X_3,
Y_1^{-1},Y_2Y_1^{-1},Y_3Y_1^{-1},Z^{-1}Y_1^{-2})$
\\
\hline
$\{\mathbf{ v}_4,\mathbf{ v}_5, \mathbf{ v}_6, \mathbf{ v}_7, 
\mathbf{ v}_8, \mathbf{ v}_9, \mathbf{ v}_{11}\}$
&
$\mbox{Spec}(\mathbb{C}[X_1^{-1},X_2^{-1},X_3^{-1},Y_1,Y_2,Y_3,
ZX_1^{-1}X_2^{-1}X_3^{-1}])$
\\
\hline
$\{\mathbf{ v}_6,\mathbf{ v}_5, \mathbf{ v}_3, \mathbf{ v}_8, 
\mathbf{ v}_9, \mathbf{ v}_{10}, \mathbf{ v}_{11}\}$
&
$\mbox{Spec}(\mathbb{C}[X_1^{-1},X_2^{-1},X_3,
Y_1^{-1},Y_2Y_1^{-1},Y_3Y_1^{-1},$
\\
&	
$ZX_1^{-1}X_2^{-1}Y_1^2])$
\\
\hline
$\{\mathbf{ v}_6,\mathbf{ v}_5, \mathbf{ v}_3, \mathbf{ v}_7, 
\mathbf{ v}_8, \mathbf{ v}_9, \mathbf{ v}_{12}\}$
&
$\mbox{Spec}(\mathbb{C}[X_1^{-1},X_2^{-1},X_3,
Y_1,Y_2,Y_3,Z^{-1}X_1X_2])$
\\
\hline
$\{\mathbf{ v}_6,\mathbf{ v}_2, \mathbf{ v}_3, \mathbf{ v}_8, 
\mathbf{ v}_9, \mathbf{ v}_{10}, \mathbf{ v}_{12}\}$
&
$\mbox{Spec}(\mathbb{C}[X_1^{-1},X_2,X_3,
Y_1^{-1},Y_2Y_1^{-1},Y_3Y_1^{-1},Z^{-1}X_1Y_1^{-2}])$
\\
\hline
$\{\mathbf{ v}_6,\mathbf{ v}_5, \mathbf{ v}_4, \mathbf{ v}_8, 
\mathbf{ v}_9, \mathbf{ v}_{10}, \mathbf{ v}_{11}\}$
&
$\mbox{Spec}(\mathbb{C}[X_1^{-1},X_2^{-1},X_3^{-1},
Y_1^{-1},Y_2Y_1^{-1},Y_3Y_1^{-1},$
\\
&
$ZX_1^{-1}X_2^{-1}X_3^{-1}Y_1^2])$
\\
\hline
$\{\mathbf{ v}_6,\mathbf{ v}_5, \mathbf{ v}_4, \mathbf{ v}_7, 
\mathbf{ v}_8, \mathbf{ v}_9, \mathbf{ v}_{12}\}$
&
$\mbox{Spec}(\mathbb{C}[X_1^{-1},X_2^{-1},X_3^{-1},
Y_1,Y_2,Y_3, Z^{-1}X_1X_2X_3])$
\\
\hline
$\{\mathbf{ v}_6,\mathbf{ v}_5, \mathbf{ v}_3, \mathbf{ v}_8, 
\mathbf{ v}_9, \mathbf{ v}_{10}, \mathbf{ v}_{12}\}$
&
$\mbox{Spec}(\mathbb{C}[X_1^{-1},X_2^{-1},X_3,
Y_1^{-1},Y_2Y_1^{-1},Y_3Y_1^{-1},$
\\
&
$Z^{-1}X_1X_2Y_1^{-2}])$
\\
\hline
$\{\mathbf{ v}_6,\mathbf{ v}_5, \mathbf{ v}_4, \mathbf{ v}_8, 
\mathbf{ v}_9, \mathbf{ v}_{10}, \mathbf{ v}_{12}\}$
&
$\mbox{Spec}(\mathbb{C}[X_1^{-1},X_2^{-1},X_3^{-1},
Y_1^{-1},Y_2Y_1^{-1},Y_3Y_1^{-1},$
\\
&
$Z^{-1}X_1X_2X_3Y_1^{-2}])$
\\
\hline
\end{tabular}
\end{center}
\caption{}
\label{tab:coordinate}
\end{table}
The other affine toric varieties unlisted in 
Table~\ref{tab:coordinate} can be obtained easily by the symmetry of $V$.

\section{Direct computation of
  $\mathcal{F}_{\mathrm{Td}^p}$}\label{sec:localization}

First, we shall make the family $\{\mathcal{F}_{\mathrm{Td}^p}\}_p$
simpler in the case of the anticanonical polarization.  For a K\"ahler
form $\omega\in c_1(V)$, let $g$ be the associated K\"ahler metric.
We have the Levi-Civita connection $\theta=g^{-1}\partial g$ and its
curvature form $\Theta=\bar{\partial}\theta$.  Then, for the
associated covariant derivative $\nabla$ with $\theta$, $L(X)$ can be
expressed by
\[
L(X)=\nabla X=\nabla_j X^i dz^j\otimes \frac{\partial}{\partial z^i}
\]
where $X\in \mathfrak{h}(V)$, because $\nabla$ is torsion free.  Now
assume that $(V, K^{-1}_V)$ is a Fano manifold with anticanonical
polarization.  By the Calabi-Yau theorem \cite{yau77}, for a K\"ahler
form $\omega\in c_1(V)$ there exists another K\"ahler form $\eta\in
c_1(V)$ whose Ricci form $\rho_\eta$ is equal to $\omega$.  For
$X\in\mathfrak{h}_0(V)$ let $\tilde{u}_{X}$ be the Hamiltonian
function with respect to $\omega$ and a different normalization 
from~(\ref{eq:normalization})
\[
\int_V \tilde{u}_X \omega^m =-f(X).
\]
Recall that $\tilde{u}_X=\Delta_\eta \tilde{u}_X$, where $\Delta_\eta$
is the Laplacian of $\eta$.  Let
\begin{align*}
{\mathcal G}_{\mathrm{Td}^p}(X) 
:= &\, (m-p+1) \int_V \mathrm{Td}^p(\Theta_\eta) \wedge \tilde{u}_X \rho_\eta^{m-p}
\\ 
&  
+ \int_V \mathrm{Td}^p(L_\eta(X) + \Theta_\eta) \wedge \rho_\eta^{m-p+1}.
\end{align*}
Here $\Theta_\eta$ is the curvature form of the Levi-Civita connection
$\theta_\eta$ with respect to $\eta$ and $L_\eta(X)$ is also
associated with $\theta_\eta$.  The proof of Theorem 3.2 in
\cite{futaki-ono-sano0811} implies that the difference between
$\mathcal{F}_{\mathrm{Td}^p}$ and $\mathcal{G}_{\mathrm{Td}^p}$ is
equal to a constant multiple of $\mathcal{F}_{\mathrm{Td}^1}$ for any
$p$.

\begin{lemma}\label{lem:localization}
Let $V$ be a Fano manifold with K\"ahler-Einstein metrics.
Then,
\begin{equation}\label{eq:td_reduction}
\mathcal{F}_{\mathrm{Td}^{p}}(X)
=
\int_V (\mathrm{Td}^{p}\cdot c_1^{m-p+1})(L_\eta(X)+\Theta_\eta)
\end{equation}
where $X\in \mathfrak{h}_0(V)$.
\end{lemma}

\begin{proof}
  Since $V$ admits K\"ahler-Einstein metrics and
  $\mathcal{F}_{\mathrm{Td}^1}$ is proportional to the original Futaki
  invariant $f$, $\mathcal{F}_{\mathrm{Td}^1}$ vanishes.  So
  $\mathcal{F}_{\mathrm{Td}^p}$ is equal to
  $\mathcal{G}_{\mathrm{Td}^p}$.  Hence, we find
\begin{align*}
\mathcal{F}_{\mathrm{Td}^{p}}(X)
& =  
(m-p+1)
\int_V \mathrm{Td}^{p}(\Theta_\eta)\wedge (\Delta_\eta u_X) \rho_\eta^{m-p}
\\
&
\quad +\int_V \mathrm{Td}^{p}(L_\eta(X)+\Theta_\eta)\wedge \rho_\eta^{m-p+1}
\\
&=
(m-p+1)\int_V \mathrm{Td}^{p}(\Theta_\eta)
\wedge c_1(L_\eta(X)) c_1(\Theta_\eta)^{m-p}
\\
&
\quad +\int_V \mathrm{Td}^{p}(L_\eta(X)+\Theta_\eta)\wedge c_1(\Theta_\eta)^{m-p+1}
\\
&=
\int_V \mathrm{Td}^{p}(L_\eta(X)+\Theta_\eta)
\wedge 
\\
&	
\quad \{(m-p+1)c_1(L_\eta(X)) c_1(\Theta_\eta)^{m-p}
+ c_1(\Theta_\eta)^{m-p+1}\}
\\
&=
\int_V (\mathrm{Td}^{p}\cdot c_1^{m-p+1})(L_\eta(X)+\Theta_\eta).
\end{align*}
\end{proof}

Since the right hand of (\ref{eq:td_reduction}) is a kind of the
integral invariants in \cite{futaki-morita85}, we can apply the
localization formula in \cite{futaki-morita85} for
$\mathcal{F}_{\mathrm{Td}^p}$ as follows.  Assume that $X$ has only
isolated zeroes $\{p_i\}$ and that $L(X)_{p_i}$ is non-degenerate at
each $p_i$, \textit{i.e.},
\[
\det(L(X)_{p_i})=\det 
\biggl(\frac{\partial X^k}{\partial z^l}(p_i) \biggr)_{1\le k,l \le m} \neq 0,
\]
where $(z^1, \cdots, z^m)$ are local coordinates.
Then we have
\begin{equation}\label{eq:td_reduction2}
\mathcal{F}_{\mathrm{Td}^p}(X)
=
\sum_{p_i}
\frac{(\mathrm{Td}^p\cdot c_1^{m-p+1}) (L(X)_{p_i})}{\det (L(X)_{p_i})}.
\end{equation}
As for the localization formula, see also \cite{futaki88}.

We consider the following one-parameter subgroup $\{\sigma_t\}$ in the
maximal torus of $\mathrm{Aut}(V)$; it is written by
\begin{align*}
&\sigma_t\cdot(X_1,X_2,X_3,
	Y_1, Y_2, Y_3,
	Z ) 
	\\
&= 
	(e^{a_1t}X_1,e^{a_2t}X_2,e^{a_3t}X_3,
	e^{b_1t}Y_1, e^{b_2t}Y_2, e^{b_3t}Y_3,
	e^{ct}Z)
\end{align*}
in the affine variety
$\mbox{Spec}(\mathbb{C}[X_1,X_2,X_3,Y_1,Y_2,Y_3,Z])$, which
corresponds to the $7$-dimensional cone generated by $\{\mathbf{
  v}_1,\mathbf{ v}_2, \mathbf{ v}_3, \mathbf{ v}_7, \mathbf{ v}_8,
\mathbf{ v}_9, \mathbf{ v}_{11}\}$.  Here, $(X_1,X_2,X_3)$ are affine
coordinates of $(\mathbb{P}^1)^3$, $(Y_1,Y_2,Y_3)$ are affine
coordinates of $\mathbb{P}^3$, and $Z$ is an affine coordinate of the
fibre.  Hence, we have
\[
	X_1=\frac{x_0}{x_1},
	X_2=\frac{x_2}{x_3},
	X_3=\frac{x_4}{x_5},
	Y_1=\frac{y_0}{y_3},
	Y_2=\frac{y_1}{y_3},
	Y_3=\frac{y_2}{y_3},
\]
where
\[
(
[x_0: x_1], [x_2: x_3], [x_4: x_5],
[y_0: y_1: y_2: y_3]
)
\]
are homogeneous coordinates of $(\mathbb{P}^1)^3\times \mathbb{P}^3$.
Let us see $\sigma_t$ in terms of another affine coordinates by using
the coordinate transformations (see Table~\ref{tab:coordinate}).  For
generic $\{a_i, b_j, c\}_{1\le i,j \le 3}$, the set of fixed points of
$\sigma_t$ consists of the following isolated $64$ points;
\begin{equation*}
\{
(\mathbf{x}_1, \mathbf{x}_2, \mathbf{x}_3, \mathbf{y}, \mathbf{z}) \in V
\mid
\mathbf{x}_i, \, \mathbf{z}=\mathbf{ p}_- \,\,\mbox{or }\mathbf{ p}_+, \,\,\,
\mathbf{y}= \mathbf{ p}_j  \,\, (j=1,2,3,4)
\},
\end{equation*}
where $\mathbf{ p}_-$ denotes $[1:0]$, $\mathbf{ p}_+$ denotes $[0:1]$
and $\mathbf{ p}_1=[1:0:0:0]$, $\mathbf{ p}_2=[0:1:0:0]$, $\mathbf{
  p}_3=[0:0:1:0]$, $\mathbf{ p}_4=[0:0:0:1]$.

Next we shall calculate $L(X)$ at each fixed point of $\sigma_t$.  For
example, let us consider $L(X)$ at
\[
(\mathbf{ p}_+, \mathbf{ p}_+, \mathbf{ p}_+, \mathbf{ p}_4, \mathbf{ p}_+)
=([0:1], [0:1], [0:1], [0:0:0:1], [0:1]).
\]
This point is the origin in the affine variety
$\mathrm{Spec}(\mathbb{C}[X_1, X_2, X_3, Y_1, \!Y_2, \!Y_3, Z])$
associated with the $7$-dimensional cone generated by 
$\{\mathbf{v}_1, \!\!\mathbf{ v}_2, \!\!\mathbf{ v}_3, \!\!\mathbf{ v}_7, 
\!\!\mathbf{ v}_8,\! \mathbf{ v}_{9}, \!\mathbf{ v}_{11}\!\}.$  
The holomorphic vector field
with respect to $\sigma_t$ around the point is expressed by
\[
\sum_{i=1}^{3} a_iX_i\frac{\partial}{\partial X_i}
+\sum_{j=1}^{3}b_iY_j\frac{\partial}{\partial Y_j}
+cZ\frac{\partial}{\partial Z}.
\]
Hence $L(X)$ at $(\mathbf{ p}_+, \mathbf{ p}_+, \mathbf{ p}_+,
\mathbf{ p}_4, \mathbf{ p}_+)$ is given by
\[
L(X)=
\mathrm{diag} (a_1, a_2, a_3, b_1, b_2, b_3, c).
\]
For another example, let us consider $L(X)$ at 
\[
(\mathbf{ p}_-, \mathbf{ p}_+, \mathbf{ p}_+, \mathbf{ p}_1, \mathbf{ p}_+)
=([1:0], [0:1], [0:1], [1:0:0:0], [0:1]).
\]
This point is the origin in 
$ \mathrm{Spec} (\mathbb{C}[X_1^{-1}\!,\! X_2,X_3, 
\!Y_1^{-1}\!,\! Y_2Y_1^{-1}\!,\! Y_3Y_1^{-1}\!,\! ZX_1^{-1}Y_1^{2}]) $
associated with the $7$-dimensional cone generated by 
$\{\!\mathbf{v}_6, \!\!\mathbf{ v}_2, \!\!\mathbf{ v}_3, 
\!\!\mathbf{ v}_8, \!\!\mathbf{ v}_9,
\!\!\mathbf{ v}_{10}, \!\!\mathbf{ v}_{11}\!\}.$  
The holomorphic vector field
with respect to $\sigma_t$ around the point is expressed by
\begin{multline*}
-a_1U_1\frac{\partial}{\partial U_1}
+
\sum_{i=2}^{3} a_iU_i\frac{\partial}{\partial U_i}
-b_1U_4\frac{\partial}{\partial U_4}
+\sum_{j=2}^3(b_{j}-b_1)U_{3+j}\frac{\partial}{\partial U_{3+j}}
+(c-a_1+2b_1)U_7\frac{\partial}{\partial U_7},
\end{multline*}
where
\begin{multline*}
U_1:=X_1^{-1}, U_2:=X_2, U_3:=X_3, 
U_4:=Y_1^{-1}, U_5:=Y_2Y_1^{-1}, U_6:=Y_3Y_1^{-1}, 
U_7:=ZX_1^{-1}Y_1^{2}.
\end{multline*}
Hence $L(X)$ at $(\mathbf{ p}_-, \mathbf{ p}_+, \mathbf{ p}_+,
\mathbf{ p}_1, \mathbf{ p}_+)$ is given by
\[
L(X)=
\mathrm{diag} (-a_1, a_2, a_3, -b_1, b_2-b_1, b_3-b_1, c-a_1+2b_1).
\]
As for the other fixed points, the computations of $L(X)$ are given by the 
Table \ref{tab:action};
\begin{table}
{\small
\begin{center}
\begin{tabular}{|c|c|l|}
	\hline
		no. & fixed pt & $L(X)$
	\\
	\hline
			1-1&
		$(+++, \mathbf{ p}_1,\pm)$ &
		$(
			a_1, a_2, a_3, 
			b_2-b_1, b_3-b_1, -b_1, \pm(c+2b_1)
		)$
	\\
	\hline
			1-2&
		$(-++, \mathbf{ p}_1,\pm)$ &
		$(
			-a_1, a_2, a_3, 
			b_2-b_1, b_3-b_1, -b_1, \pm(c-a_1+2b_1)
		)$
	\\
	\hline
			1-3&
		$(+-+, \mathbf{ p}_1,\pm)$ &
		$(
			a_1, -a_2, a_3, 
			b_2-b_1, b_3-b_1, -b_1, \pm(c-a_2+2b_1)
		)$
	\\
	\hline
			1-4&
		$(++-, \mathbf{ p}_1,\pm)$ &
		$(
			a_1, a_2, -a_3, 
			b_2-b_1, b_3-b_1, -b_1, \pm(c-a_3+2b_1)
		)$
	\\
	\hline
			1-5&
		$(+--, \mathbf{ p}_1,\pm)$ &
		$(
			a_1, \!-a_2, \!-a_3, 
			b_2-b_1, b_3-b_1, \!-b_1, \pm(c-a_2-a_3+2b_1)
		)$
	\\
	\hline
			1-6&
		$(-+-, \mathbf{ p}_1,\pm)$ &
		$(
			-a_1, a_2, \!-a_3, 
			b_2\!-b_1, b_3\!-b_1, \!-b_1, \pm(c\!-a_1\!-a_3\!+2b_1)
		)$
	\\
	\hline
			1-7&
		$(--+, \mathbf{ p}_1,\pm)$ &
		$(
			-a_1, \!-a_2, a_3, 
			b_2\!-b_1, b_3\!-b_1, \!-b_1, \pm(c-a_1-a_2+2b_1)
		)$
	\\
	\hline
			1-8&
		$(---, \mathbf{ p}_1,\pm)$ &
		$(
			-a_1, \!-a_2, \!-a_3, 
			b_2\!-b_1, b_3\!-b_1, \!-b_1, 
                        \pm\big(c\!-\sum a_i+2b_1)
		\big)$
	\\
	\hline
			2-1&
		$(+++, \mathbf{ p}_2,\pm)$ &
		$(
			a_1, a_2, a_3, 
			b_1-b_2, b_3-b_2, -b_2, \pm(c+2b_2)
		)$
	\\
	\hline
			2-2&
		$(-++, \mathbf{ p}_2,\pm)$ &
		$(
			-a_1, a_2, a_3, 
			b_1-b_2, b_3-b_2, -b_2, \pm(c-a_1+2b_2)
		)$
	\\
	\hline
			2-3&
		$(+-+, \mathbf{ p}_2,\pm)$ &
		$(
			a_1, -a_2, a_3, 
			b_1-b_2, b_3-b_2, -b_2, \pm(c-a_2+2b_2)
		)$
	\\
	\hline
			2-4&
		$(++-, \mathbf{ p}_2,\pm)$ &
		$(
			a_1, a_2, -a_3, 
			b_1-b_2, b_3-b_2, -b_2, \pm(c-a_3+2b_2)
		)$
	\\
	\hline
			2-5&
		$(+--, \mathbf{ p}_2,\pm)$ &
		$(
			a_1, \!-a_2, \!-a_3, 
			b_1\!-b_2, b_3\!-b_2, \!-b_2, \pm(c-a_2-a_3+2b_2)
		)$
	\\
	\hline
			2-6&
		$(-+-, \mathbf{ p}_2,\pm)$ &
		$(
			-a_1, a_2, \!-a_3, 
			b_1\!-b_2, b_3\!-b_2, \!-b_2, \pm(c-a_1-a_3+2b_2)
		)$
	\\
	\hline
			2-7&
		$(--+, \mathbf{ p}_2,\pm)$ &
		$(
			-a_1, \!-a_2, a_3, 
			b_1\!-b_2, b_3\!-b_2, \!-b_2, \pm(c-a_1-a_2+2b_2)
		)$
	\\
	\hline
			2-8&
		$(---, \mathbf{ p}_2,\pm)$ &
		$(
			-a_1, \!-a_2, \!-a_3, 
			b_1\!-b_2, b_3\!-b_2, \!-b_2, 
                        \pm\big(c-\sum a_i+2b_2)
		\big)$	
	\\
	\hline
			3-1&
		$(+++, \mathbf{ p}_3,\pm)$ &
		$(
			a_1, a_2, a_3, 
			b_1-b_3, b_2-b_3, -b_3, \pm(c+2b_3)
		)$
	\\
	\hline
			3-2&
		$(-++, \mathbf{ p}_3,\pm)$ &
		$(
			-a_1, a_2, a_3, 
			b_1-b_3, b_2-b_3, -b_3, \pm(c-a_1+2b_3)
		)$
	\\
	\hline
			3-3&
		$(+-+, \mathbf{ p}_3,\pm)$ &
		$(
			a_1, -a_2, a_3, 
			b_1-b_3, b_2-b_3, -b_3, \pm(c-a_2+2b_3)
		)$
	\\
	\hline
			3-4&
		$(++-, \mathbf{ p}_3,\pm)$ &
		$(
			a_1, a_2, -a_3, 
			b_1-b_3, b_2-b_3, -b_3, \pm(c-a_3+2b_3)
		)$
	\\
	\hline
			3-5&
		$(+--, \mathbf{ p}_3,\pm)$ &
		$(
			a_1, \!-a_2, \!-a_3, 
			b_1\!-b_3, b_2\!-b_3, \!-b_3, \pm(c-a_2-a_3+2b_3)
		)$
	\\
	\hline
			3-6&
		$(-+-, \mathbf{ p}_3,\pm)$ &
		$(
			-a_1, a_2, \!-a_3, 
			b_1\!-b_3, b_2\!-b_3, \!-b_3, \pm(c-a_1-a_3+2b_3)
		)$
	\\
	\hline
			3-7&
		$(--+, \mathbf{ p}_3,\pm)$ &
		$(
			-a_1, \!-a_2, a_3, 
			b_1\!-b_3, b_2\!-b_3, \!-b_3, \pm(c-a_1-a_2+2b_3)
		)$
	\\
	\hline
			3-8&
		$(---, \mathbf{ p}_3,\pm)$ &
		$(
			-a_1, \!-a_2, \!-a_3, 
			b_1\!-b_3, b_2\!-b_3, \!-b_3, 
                        \pm\big(c-\sum a_i+2b_3)
		\big)$	
	\\
	\hline
			4-1&
		$(+++, \mathbf{ p}_4,\pm)$ &
		$(
			a_1, a_2, a_3, 
			b_1, b_2, b_3, \pm c
		)$
	\\
	\hline
			4-2&
		$(-++, \mathbf{ p}_4,\pm)$ &
		$(
			-a_1, a_2, a_3, 
			b_1, b_2, b_3, \pm(c-a_1)
		)$
	\\
	\hline
			4-3&
		$(+-+, \mathbf{ p}_4,\pm)$ &
		$(
			a_1, -a_2, a_3, 
			b_1, b_2, b_3, \pm(c-a_2)
		)$
	\\
	\hline
			4-4&
		$(++-, \mathbf{ p}_4,\pm)$ &
		$(
			a_1, a_2, -a_3, 
			b_1, b_2, b_3, \pm(c-a_3))
		)$
	\\
	\hline
			4-5&
		$(+--, \mathbf{ p}_4,\pm)$ &
		$(
			a_1, -a_2, -a_3, 
			b_1, b_2, b_3, \pm(c-a_2-a_3)
		)$
	\\
	\hline
			4-6&
		$(-+-, \mathbf{ p}_4,\pm)$ &
		$(
			-a_1, a_2, -a_3, 
			b_1, b_2, b_3, \pm(c-a_1-a_3)
		)$
	\\
	\hline
			4-7&
		$(--+, \mathbf{ p}_4,\pm)$ &
		$(
			-a_1, -a_2, a_3, 
			b_1, b_2, b_3, \pm(c-a_1-a_2)
		)$
	\\
	\hline
			4-8&
		$(---, \mathbf{ p}_4,\pm)$ &
		$(
			-a_1, -a_2, -a_3, 
			b_1, b_2, b_3, \pm\big(c-\sum a_i)
		\big)$	
	\\
	\hline
\end{tabular}
\end{center}
\caption{}
\label{tab:action}
}
\end{table}
As for the notation of the column of fixed points, for example, $(+++,
\mathbf{ p}_1,-)$ means a fixed point $(\mathbf{ p}_+,\mathbf{
  p}_+,\mathbf{ p}_+,\mathbf{ p}_1,\mathbf{
  p}_-)=([0:1],[0:1],[0:1],[1:0:0:0],[1:0])$.  Remark that $L(X)$ is
non-degenerate at any fixed point for generic $\{a_i, b_j,
c\}_{i,j=1,2,3}$.

Finally, we shall list below the results of calculations of
$\mathcal{F}_{\mathrm{Td}^{p}}$ ($2\le p \le 7$) with respect to the
holomorphic vector field induced by $\{\sigma_t\}$ for generic $\{a_i,
b_j, c\}_{i,j=1,2,3}$.  As for $p=1$, it suffices to consider $f$
instead of $\mathcal{F}_{\mathrm{Td}^1}$.  We have the localization
formula for $f$ independently of Lemma~\ref{lem:localization}
(\textit{cf.}~\cite{futaki88}).  The formula for $f$ is same as
(\ref{eq:td_reduction2}), but it holds without assuming the existence
of K\"ahler-Einstein metrics.  By using it (not Lemma
\ref{lem:localization}) we can prove that $f$ vanishes on
$\mathfrak{h}_0(V)$.  See Appendix for the calucation.  Combining this
and \cite{wang-zhu01}, we can prove that $V$ admits K\"ahler-Einstein
metrics independently of \cite{nill-paffenholz090505}.  Then, we apply
Lemma~\ref{lem:localization} to $\mathcal{F}_{\mathrm{Td}^p}$ ($p\ge
2$).  Since the computations are quite enormous, we use the computer
algebra system ``Maxima''.\footnote{Maxima is available from
  http://maxima.sourceforge.net/.}  However, in order to see that $V$
is a counterexample to Problem \ref{problem:futaki-ono-sano}, it is
sufficient to check that $\mathcal{F}_{\mathrm{Td}^p}$ does not vanish
for some $\{a_i, b_j, c\}$ and some $2 \le p \le 7$.  It is still
tough, but it would be able to check without computer.  For the
readers convenience, we put all the data needed to compute in the case
where $(a_1, a_2, a_3, b_1, b_2, b_3, c)=(1,1,1,1,2,3,4)$ and $p=2$ in
Appendix.

\begin{description}
\item[$p=2$]
\begin{align}
		12\mathcal{F}_{\mathrm{Td}^{2}}(X)
	&=
		\sum_{\mathbf{q}: \,\,\mathrm{fixed} \,\mathrm{pt}}
		\frac{(c_1^2+c_2)c_1^6(L(X)_{\mathbf{q}})}{\det(L(X)_{\mathbf{q}})}
\nonumber
	\\
	&=
		\sum_{\mathbf{q}: \,\,\mathrm{fixed} \,\mathrm{pt}}
		\frac{(c_2c_1^6)(L(X)_{\mathbf{q}})}{\det(L(X)_{\mathbf{q}})}
\nonumber
	\\
	\label{eq:check_coefficient_1}
	&=
		13056\Big(\sum a_i- \sum b_i -2c\Big).
\end{align}
\item[$p=3$]
\begin{align*}
		24\mathcal{F}_{\mathrm{Td}^{3}}(X)
	&=
		\sum_{\mathbf{q}: \,\,\mathrm{fixed} \,\mathrm{pt}}
		\frac{c_2c_1^6(L(X)_{\mathbf{q}})}{\det(L(X)_{\mathbf{q}})}
	\\
	&=
			12\mathcal{F}_{\mathrm{Td}^{2}}(X)	
	\\
	&=
		13056\Big(\sum a_i- \sum b_i -2c\Big).
\end{align*}
\item[$p=4$]
\begin{align*}
		720\mathcal{F}_{\mathrm{Td}^{4}}(X)
	&=
		\sum_{\mathbf{q}: \,\,\mathrm{fixed} \,\mathrm{pt}}
		\frac{(-c_1^4+4c_1^2c_2+3c_2^2+c_1c_3-c_4)
			c_1^4(L(X)_{\mathbf{q}})}{\det(L(X)_{\mathbf{q}})}
	\\
	&=
		94080\Big(\sum a_i- \sum b_i -2c\Big).
\end{align*}
\item[$p=5$]
\begin{align*}
		1440\mathcal{F}_{\mathrm{Td}^{5}}(X)
	&=
		\sum_{\mathbf{q}: \,\,\mathrm{fixed} \,\mathrm{pt}}
		\frac{(-c_1^3c_2+3c_1c_2^2+c_1^2c_3-c_1c_4)
			c_1^3(L(X)_{\mathbf{q}})}{\det(L(X)_{\mathbf{q}})}
	\\
	&=
		28800\Big(\sum a_i- \sum b_i -2c\Big).
\end{align*} 
\item[$p=6$]
\begin{align*}
		60480\mathcal{F}_{\mathrm{Td}^{6}}(X)
	&=
		\sum_{\mathbf{q}: \,\,\mathrm{fixed} \,\mathrm{pt}}
		\biggl(\frac{(2c_1^6-12c_1^4c_2+11c_1^2c_2^2+10c_3^2+5c_1^3c_3
		)
			c_1^2(L(X)_{\mathbf{q}})}{\det(L(X)_{\mathbf{q}})}
	\\
	&
	\quad	+\frac{(11c_1c_2c_3-c_3^2-5c_1^2c_4-9c_2c_4-2c_1c_5+2c_6)
			c_1^2(L(X)_{\mathbf{q}})}{\det(L(X)_{\mathbf{q}})}\biggr)
	\\
	&=
		82176\Big(\sum a_i- \sum b_i -2c\Big).
\end{align*}

\item[$p=7$]
\begin{align*}
		120960\mathcal{F}_{\mathrm{Td}^{7}}(X)
	&=
		\sum_{\mathbf{q}: \,\,\mathrm{fixed} \,\mathrm{pt}}
		\biggl(\frac{(11c_1^2c_2c_3-9c_1c_2c_4+2c_1c_6-2c_1^2c_5
		)
			c_1(L(X)_{\mathbf{q}})}{\det(L(X)_{\mathbf{q}})}
	\\
	&
	\quad	
    +\frac{(2c_1^3c_4-c_1c_3^2-2c_1^4c_3+10c_1c_2^3-10c_1^3c_2^2+2c_1^5c_2)
			c_1(L(X)_{\mathbf{q}})}{\det(L(X)_{\mathbf{q}})}\biggr)
	\\
	&=
		16128\Big(\sum a_i- \sum b_i -2c\Big).
\end{align*}
\end{description}
Remark that all $\mathcal{F}_{\mathrm{Td}^{p}}$ ($2\le p \le7$) are
proportional to each other.  This result is consistent with the fact
that $\dim N_\mathbb{R}^{\mathcal{W}(V)}=1$.  Therefore we can
conclude that even if a Fano manifold admits K\"ahler-Einstein metrics
(\textit{i.e.}, cscK metrics),
$\{\mathcal{F}_{\mathrm{Td}^{p}}\}_{p=1,\ldots,m}$ may not vanish.
The proof of the main theorem is completed.

\section{The derivatives of the Hilbert series}\label{sec:hilbert}

In \cite{futaki-ono-sano0811}, Futaki and the first two authors showed
a relation between the obstructions to asymptotic Chow semistability
and the derivatives of the Hilbert series.  In the present section, we
will see that we can also show Theorem \ref{thm:main} using such
relation.

We first review the definition and some properties of the Hilbert
series.  See \cite{futaki-ono-sano0811} for more details.  Let $V$ be
a toric Fano $m$-fold and $Q$ be the corresponding Fano polytope.  The
polar dual $P$ of $Q$, which is the Delzant polytope of $(V,K_V^{-1})$
in $M_{\mathbb{R}}\simeq \mathbb{R}^m$, is defined as
$$ P:=\{w\in \mathbb{R}^m\ |\ \langle \mathbf{v}_j , w \rangle\ge -1\}$$       
where $\mathbf{v}_j \in \mathbb{Z}^m$ is a vertex of $Q$ for each $j$.

We call the convex rational polyhedral cone
$$ \mathcal C^{\ast}:=\{y \in \mathbb{R}^{m+1}\ |\ 
\langle\lambda _j , y \rangle\ge 0\}$$ 
the toric diagram of $V$, where
$\lambda_j =(\mathbf{v}_j,1)\in \mathbb{Z}^{m+1}$.  Note here that
this cone is a pointed cone in $\mathbb{R}^{m+1}$, that is to say,
$\mathcal{C}^*\cap (-\mathcal{C}^*)=\{0\}$.  We can also represent
$\mathcal{C}^*$ by
$$\mathcal{C}^*=\{c_1\boldsymbol{\mu}_1+\cdots +
c_k\boldsymbol{\mu}_k\, | \,
c_i \ge 0, \,\, i=1,\cdots, k
\}$$
where ${\boldsymbol{\mu}}_j=(\w_j,1)\in \mathbb{Z}^{m+1}$ and
$\w_1,\cdots,\w_k$
are the vertices of the Delzant polytope $P$.
Then we can define the
(multi-graded) Hilbert series $C({\bf{x},\mathcal C^*})$ of
the rational cone $\mathcal{C}^*$
by
$$
C({\bf{x}},\mathcal C^{\ast})
=\sum_{{\bf{a}}\in \mathcal C^{\ast}\cap \mathbb{Z}^{m+1}}{\bf x}^{\bf a}\ \ \ \ \
({\bf x}^{\bf a}=x_1^{a_1}\cdots x_{m+1}^{a_{m+1}}).
$$
As proved in \cite{miller-sturmfels05}, the Hilbert series
$C({\bf{x}},\mathcal C^*)$ can be written as a rational generating
function of the form
$$C({\bf{x}}, \mathcal C^{\ast})
=\frac{K_{\mathcal C^{\ast}}({\bf{x}})}
{(1-{\bf {x}}^{{\boldsymbol{\mu}}_1})\cdots 
(1-{\bf {x}}^{{\boldsymbol{\mu}}_k})}$$
where $K_{\mathcal C^{\ast}}({\bf{x}})$
is a Laurent polynomial. 
Using Brion's formula, we are able to
calculate the right hand side of the above equation as follows, see
\cite{futaki-ono-sano0811};
$$C({\bf{x}},\mathcal C^{\ast})
=\sum_{j=1}^k\frac{1}{1-{\bf {x}}^{{\boldsymbol{\mu}}_j}}\prod _{b=1}^m
\frac{1}{(1-\Tilde{\bf x}^{{\bf e}_{j,b}})},$$  
where
${\bf{e}}_{j,1},\cdots,{\bf{e}}_{j,m}\in \mathbb{Z}^m$ 
denote the generators of the edges emanating from a vertex 
$\w_j$ and $\Tilde{\bf x}=(x_1,\cdots ,x_m)$.

Let $\mathcal{C}_R$ be the convex polytope defined as
$$\mathcal{C}_R=\{\xi\in \mathcal{C}\, |\,
\xi=({\bf b},m+1)\},$$ 
where $\mathcal{C}$ is the interior of the dual
cone of $\mathcal{C}^*$.  For $\xi=({\bf{b}},m+1)\in \mathcal{C}_R$ we
write
$${\bf{e}}^{-t\xi}=(e^{-b_1t},\cdots,e^{-b_m t},e^{-(m+1)t})$$  
and consider
$$
C({\bf{e}}^{-t\xi},\mathcal C^{\ast})=
\frac{K_{\mathcal C^{\ast}}({\bf{e}}^{-t\xi})}
{(1-e^{-t\langle {\boldsymbol{\mu}}_1,\xi\rangle})\cdots (1-
e^{-t\langle {\boldsymbol{\mu}}_k,\xi\rangle})}.
$$
For each fixed $\xi\in \mathcal{C}_R$, the Laurent expansion of
$C({\bf{e}}^{-t\xi},\mathcal C^{\ast})$ at $t=0$ is written as
$$
C({\bf{e}}^{-t\xi},\mathcal C^{\ast})
=\frac{C_{-m-1}(\bf{b})}{t^{m+1}}+\frac{C_{-m}(\bf{b})}{t^{m}}
+\frac{C_{-m+1}(\bf{b})}{t^{m-1}}+\cdots.$$  

In \cite{futaki-ono-sano0811}, it was shown that the following
relation between the invariants $\mathcal{F}_{\text{\upshape Td}^p}$
and the derivatives of $C_i(\bf{b})$ at $\bf{b}=0$.

\begin{theorem}[\cite{futaki-ono-sano0811}]\label{thm:41}
  The linear span of the derivatives $d_{0}C_i(\bf{b})$,
  $i=-m-1,-m,\dots$, coincides with the linear span of
  $\mathcal{F}_{\text{\upshape Td}^1},\dots
  \mathcal{F}_{\text{\upshape Td}^m}$ restricted to
  $\mathfrak{t}\otimes \mathbb{C}\simeq \mathbb{C}^m$.
\end{theorem}

Let $V$ be the toric Fano manifold defined by
(\ref{eq:vertices_fano_polytope}) and $P$ be the Delzant polytope of
$(V, K_V^{-1})$. Then we can calculate $\{\w_j,
{\bf{e}}_{j,b}\}_{(1\leq j\leq 64,\; 1\leq b\leq 7)}$ explicitly from
the argument in Section 2, and so we see whether all
$\mathcal{F}_{Td^p}$ vanish or not. Note that these 64 verticies of
$P$ correspond to the facets of the Fano polytope defined by
(\ref{eq:vertices_fano_polytope}). However it is difficult to check it
directly, because in our case the Hilbert series has too many terms.

To solve this problem, we use the following proposition.

\begin{proposition}\label{prop:42}
If
$\mathcal{F}_{\text{\upshape Td}^p}=0$
for each $p=1,\dots,m$ then
\begin{equation}\label{eq:hil}
\frac{\partial}{\partial x}C(x^{n_1},\dots,x^{n_m},e^{-(m+1)t})
_{|x=1}=0
\end{equation}
for any ${\bf n}=(n_1,\dots,n_m)\in \mathbb{N}^m$.
\end{proposition}

\begin{proof}
If $\mathcal{F}_{\text{\upshape Td}^p}=0$
for each $p=1,\dots,m$ then
$$
\frac{\partial}{\partial b_i}C(e^{-b_1t},\dots,e^{-b_mt},e^{-(m+1)t})
_{|{\bf b}=0}=0
$$
holds for each $i=1,\dots,m$ by Theorem \ref{thm:41}.
Hence we easily see the proposition by the chain rule.
\end{proof}

The left hand side of \eqref{eq:hil} for the toric Fano 7-fold
associated with (\ref{eq:vertices_fano_polytope}) is computable with
computer. The combinatorial data we need is in Appendix.  For example,
in consequence of the Maple calculation, we can see
$$
\frac{\partial}{\partial x}C(x^{n_1},\cdots ,x^{n_7},e^{-8t})_{|x=1}
=
-\frac{184e^{-8t}(2e^{-32t}+31e^{-24t}+70e^{-16t}+31e^{-8t}+2)}{(-1+e^{-8t})^7}
\neq 0
$$
for $(n_1,n_2,n_3,n_4,n_5,n_6,n_7)=(1,2,3,4,5,6,7).$

Thus, by Proposition \ref{prop:42},
Theorem~\ref{thm:main} has been proved.

\section{Appendix}
\subsection{Combinatorial data of the Nill-Paffenholz's example}

In this subsection we shall list up the necessary combinatorial data
of Nill-Paffenholz's example for the calculation in
Section~\ref{sec:hilbert}.

$\bullet $ The vertices of the Fano polytope $Q$ are given by
(\ref{eq:vertices_fano_polytope}).

$\bullet $ The $64$ vertices of the polar polytope $P$ are given by
\begin{equation*}
\begin{aligned}
 (\w_1,\w_2,\w_3,\cdots,\w_{64}) \hspace{8.0cm}\\
 =\left(
\begin{array}
{@{\,}cccccccccccccccccccccccccccccccccccccccccccccccccccccccccccccccc@{\,}}
\!\!-1 & \!\!-1 & \!\!-1 & \!\!-1 & \!\!-1 & \!\!-1 & \!\!2 & \!\!-1 
& \!\!-1 & \!\!-1 & \!\!2  & \!\!-1 & \!\!-1 & \!\!-1 & \!\!2 & \!\!-1  \\  
\!\!-1 & \!\!-1 & \!\!-1 & \!\!-1 & \!\!-1 & \!\!2  & \!\!-1& \!\!-1 
& \!\!-1 & \!\!2  & \!\!-1 & \!\!-1 & \!\!-1 & \!\!2  & \!\!-1& \!\!-1  \\                              
\!\!-1 & \!\!-1 & \!\!-1 & \!\!-1 & \!\!2  & \!\!-1 & \!\!-1& \!\!-1 
& \!\!2  & \!\!\!\!-1 & \!\!-1 & \!\!-1 &  \!\!2 & \!\!-1 & \!\!-1& \!\!-1  \\                         
\!\!-1 & \!\!-1 & \!\!-1 & \!\!1  & \!\!-1 & \!\!-1 & \!\!-1& \!\!-1 
& \!\!-1 & \!\!-1 & \!\!-1 & \!\!-1 & \!\!-1 & \!\!-1 & \!\!-1& \!\!-1  \\                          
\!\!-1 & \!\!1  & \!\!-1 & \!\!-1 & \!\!-1 & \!\!-1 & \!\!-1& \!\!-1 
& \!\!1  &  \!\!1 & \!\!1  & \!\!5  & \!\!-1 & \!\!-1 & \!\!-1& \!\!-1  \\                              
\!\!-1 & \!\!-1 & \!\!1  & \!\!-1 & \!\!-1 & \!\!-1 & \!\!-1& \!\!-1 
& \!\!-1 & \!\!-1 & \!\!-1 & \!\!-1 &  \!\!1 &  \!\!1 &  \!\!1&  \!\!5  \\                          
\!\!-1 & \!\!-1 & \!\!-1 & \!\!-1 & \!\!-1 & \!\!-1 & \!\!-1& \!\!1  
& \!\!-1 & \!\!-1 & \!\!-1 &  \!\!1 & \!\!-1 & \!\!-1 & \!\!-1&  \!\!1  \\ 
\end{array}
\right.
\end{aligned}
\end{equation*}
\begin{equation*}
\begin{aligned}
\left.\begin{array}
{@{\,}cccccccccccccccccccccccccccccccccccccccccccccccccccccccccccccccc@{\,}}
\!\!-1&\!\!-1 &  \!\!2 & \!\!-1 & \!\!-1 & \!\!  2 & \!\! -1 & \!\! 2 
& \!\! -1 & \!\! 0  & \!\! 2  & \!\! 2  & \!\! 2  & \!\!  2 & \!\! 2  
& \!\! 2 & \!\! -1 & \!\!  0 \\  
-1& \!\! 2 & \!\! -1 & \!\! -1 & \!\!  2 & \!\! -1 & \!\! -1 & \!\! 2 
& \!\! 0  & \!\! -1 & \!\! 2  & \!\! 2  & \!\! 2  & \!\!  2 & \!\! 2  
& \!\! -1& \!\!  2 & \!\!  0 \\  
2& \!\!-1 & \!\! -1 & \!\! -1 & \!\!  2 & \!\!  2 & \!\!  0 & \!\! -1
& \!\! -1 & \!\! -1 & \!\! 2  & \!\! 2  & \!\! 2  & \!\!  2 & \!\! -1 
& \!\!  2& \!\!  2 & \!\!  0 \\                         
1 & \!\! 1 & \!\!  1 & \!\!  5 & \!\! -1 & \!\! -1 & \!\! -1 & \!\! -1
& \!\! -1 & \!\! -1 & \!\!  1 & \!\! -1 & \!\! -1 & \!\! -1 & \!\!  1 
& \!\!  1& \!\!  1 & \!\!  5 \\                          
-1& \!\!-1 & \!\! -1 & \!\! -1 & \!\! -1 & \!\! -1 & \!\! -1 & \!\! -1
& \!\! -1 & \!\! -1 & \!\! -1 & \!\! -1 & \!\! -1 & \!\!  1 & \!\! -1 
& \!\! -1& \!\! -1 & \!\!  -1\\                              
-1& \!\!-1 & \!\! -1 & \!\! -1 & \!\! -1 & \!\! -1 & \!\! -1 & \!\! -1
& \!\! -1 & \!\! -1 & \!\! -1 & \!\! -1 & \!\!  1 & \!\! -1 & \!\! -1 
& \!\! -1& \!\! -1 & \!\!  -1\\                          
-1& \!\!-1 & \!\! -1 & \!\!  1 & \!\! -1 & \!\! -1 & \!\!  1 & \!\! -1
& \!\! 1  & \!\!  1 & \!\! -1 & \!\! -1 & \!\! -1 & \!\! -1 & \!\! -1 
& \!\! -1& \!\! -1 & \!\!   1\\ 
\end{array}
\right.
\end{aligned}
\end{equation*}
\begin{equation*}
\begin{aligned}
\left.\begin{array}
{@{\,}cccccccccccccccccccccccccccccccccccccccccccccccccccccccccccccccc@{\,}}
0& \!\! 2 & \!\!  2 & \!\! -1 & \!\!  0 & \!\!  2 & \!\!  2 & \!\! -1
& \!\! 0  & \!\! 0  & \!\! -1 & \!\! -1 & \!\!  0 & \!\! -1 & \!\! -1 
& \!\! 0 & \!\! -1 & \!\!  -1 \\  
0& \!\! 2 & \!\! -1 & \!\!  2 & \!\!  0 & \!\!  2 & \!\! -1 & \!\! 2 
& \!\! 0  & \!\! 0  & \!\! -1 & \!\!  0 & \!\! -1 & \!\! -1 & \!\! 0  
& \!\! -1& \!\! -1 & \!\!  0 \\                              
0& \!\!-1 & \!\!  2 & \!\!  2 & \!\!  0 & \!\! -1 & \!\!  2 & \!\!  2
& \!\! 0  & \!\! -1 & \!\! 0  & \!\! -1 & \!\! -1 & \!\!  0 & \!\! -1 
& \!\! -1& \!\!  0 & \!\! -1 \\                         
-1& \!\!-1 & \!\! -1 & \!\! -1 & \!\! -1 & \!\! -1 & \!\! -1 & \!\! -1
& \!\! -1 & \!\! 5  & \!\! -1 & \!\! -1 & \!\! -1 & \!\! -1 & \!\! -1 
& \!\! -1& \!\!  5 & \!\!  5 \\                          
-1& \!\!-1 & \!\! -1 & \!\! -1 & \!\! -1 & \!\!  1 & \!\!  1 & \!\!  1
& \!\! 5  & \!\! -1 & \!\!  5 & \!\!  5 & \!\!  5 & \!\! -1 & \!\! -1 
& \!\! -1& \!\! -1 & \!\!  -1\\                              
 \!\!-1& \!\! 1 & \!\!  1 &   \!\!1 &  \!\! 5 &  \!\!-1 &  \!\!-1 
&  \!\!-1&  \!\!-1 &  \!\!-1 &  \!\!-1 &  \!\!-1 &  \!\!-1 
&   \!\!5 &   \!\!5 &   \!\!5&  \!\!-1 &   \!\!-1\\                          
 \!\!1& \!\!-1 &  \!\!-1 & \!\! -1 &   \!\!1 &  \!\!-1 &  \!\!-1 
&  \!\!-1&  \!\!1  &  1 &  1 &  1 &  1 &  1 &  1 
&   \!\!1&   \!\!1 &    \!\!1\\ 
\end{array}
\right.
\end{aligned}
\end{equation*}
\begin{equation*}
\left.\begin{array}
{@{\,}cccccccccccccccccccccccccccccccccccccccccccccccccccccccccccccccc@{\,}}
0 & \!\! -1 & \!\!  0 & \!\!  0 & \!\!  0 & \!\!  0 & \!\! 0 & \!\!  0 
& \!\!  0 & \!\! -1 & \!\! -1  & \!\! -1 \\   
-1 & \!\!  0 & \!\! -1 & \!\!  0 & \!\! -1 & \!\! -1 & \!\! -1& \!\!  0 
& \!\!  0 & \!\!  0 & \!\!  0  & \!\!  0 \\                              
-1 & \!\!  0 & \!\!  0 & \!\! -1 & \!\!  0 & \!\!  0 & \!\!  0& \!\! -1 
& \!\! -1 & \!\!  0 & \!\!  0  & \!\!  0 \\                         
5 & \!\! -1 & \!\! -1 & \!\! -1 & \!\!  5 & \!\! -1 & \!\! -1& \!\! -1 
& \!\! -1 & \!\! -1 & \!\! -1  & \!\!  5 \\                          
-1 & \!\! -1 & \!\! -1 & \!\! -1 & \!\! -1 & \!\!  5 & \!\! -1& \!\! -1 
& \!\! 5  & \!\!  5 & \!\! -1  & \!\! -1 \\                              
-1 & \!\! -1 & \!\! -1 & \!\! -1 & \!\! -1 & \!\! -1 & \!\!  5& \!\! 5  
& \!\! -1 & \!\! -1 & \!\!  5  & \!\! -1 \\                          
1 & \!\!  1 & \!\!  1 & \!\!  1 & \!\!  1 & \!\!  1 & \!\!  1& \!\! 1  
& \!\!  1 & \!\!  1 & \!\!  1  & \!\!  1 \\
\end{array}
\right).
\end{equation*}

$\bullet $ The neighbors of each vertex of $P$ are listed in 
Table~\ref{tab:neighbor} below. Here, vertices $v$ and $u$ of $P$ are called
neighbors if the interval $[u,v]$ is an edge of $P$.
\begin{table}[b]
\begin{center}
\begin{tabular}{|c|c|c|}
	\hline
	vertex & associated cone & neighbors 
	\\
	\hline
        $\w_1$
        &
		$\{\mathbf{ v}_1,\mathbf{ v}_2, \mathbf{ v}_3, \mathbf{ v}_7, 
		\mathbf{ v}_8, \mathbf{ v}_9, \mathbf{ v}_{11}\}$
        &
	         $\w_2, \w_3, \w_4, \w_5, \w_6, \w_7, \w_8$
   \\    
	\hline
        $\w_3$
        &
		$\{\mathbf{ v}_1,\mathbf{ v}_2, \mathbf{ v}_3, \mathbf{ v}_7, 
		\mathbf{ v}_8, \mathbf{ v}_{10}, \mathbf{ v}_{11}\}$
        &
        $\w_1, \w_2, \w_4, \w_{13}, \w_{14}, \w_{15}, \w_{16}$
   \\
        \hline
        $\w_7$
        &
		$\{\mathbf{ v}_6,\mathbf{ v}_2, \mathbf{ v}_3, \mathbf{ v}_7, 
		\mathbf{ v}_8, \mathbf{ v}_9, \mathbf{ v}_{11}\}$
        &
        $\w_1, \w_{11}, \w_{15}, \w_{19}, \w_{22}, \w_{24}, \w_{26}$
        \\
		\hline
        $\w_8$
        &
		$\{\mathbf{ v}_1,\mathbf{ v}_2, \mathbf{ v}_3, \mathbf{ v}_7, 
		\mathbf{ v}_8, \mathbf{ v}_9, \mathbf{ v}_{12}\}$
        &
        $\w_1, \w_{12}, \w_{16}, \w_{20}, \w_{23}, \w_{25}, \w_{26}$
   \\
   \hline
        $\w_{19}$
        &
		$\{\mathbf{ v}_6,\mathbf{ v}_2, \mathbf{ v}_3, \mathbf{ v}_8, 
		\mathbf{ v}_9, \mathbf{ v}_{10}, \mathbf{ v}_{11}\}$
        &
        $\w_4, \w_{7}, \w_{11}, \w_{15}, \w_{31}, \w_{32}, \w_{53}$
   \\  
   \hline
       $\w_{20}$
        &
		$\{\mathbf{ v}_1,\mathbf{ v}_2, \mathbf{ v}_3, \mathbf{ v}_8, 
		\mathbf{ v}_9, \mathbf{ v}_{10}, \mathbf{ v}_{12}\}$
        &
        $\w_4, \w_{8}, \w_{12}, \w_{16}, \w_{51}, \w_{52}, \w_{53}$
   \\
   \hline
        $\w_{24}$
        &
		$\{\mathbf{ v}_6,\mathbf{ v}_5, \mathbf{ v}_3, \mathbf{ v}_7, 
		\mathbf{ v}_8, \mathbf{ v}_9, \mathbf{ v}_{11}\}$
        &
        $\w_6, \w_{7}, \w_{28}, \w_{31}, \w_{36}, \w_{40}, \w_{56}$
  \\
		\hline
        $\w_{26}$
        &
		$\{\mathbf{ v}_6,\mathbf{ v}_2, \mathbf{ v}_3, \mathbf{ v}_7, 
		\mathbf{ v}_8, \mathbf{ v}_9, \mathbf{ v}_{12}\}$
        &
        $\w_7, \w_{8}, \w_{47}, \w_{50}, \w_{53}, \w_{55}, \w_{56}$
   \\
        \hline
        $\w_{27}$
        &
		$\{\mathbf{ v}_6,\mathbf{ v}_5, \mathbf{ v}_4, \mathbf{ v}_8, 
		\mathbf{ v}_9, \mathbf{ v}_{10}, \mathbf{ v}_{11}\}$
        &
        $\w_{28}, \w_{29}, \w_{30}, \w_{31}, \w_{32}, \w_{33}, \w_{34}$
   \\ 
		\hline
        $\w_{28}$
        &
		$\{\mathbf{ v}_6,\mathbf{ v}_5, \mathbf{ v}_4, \mathbf{ v}_7, 
		\mathbf{ v}_8, \mathbf{ v}_9, \mathbf{ v}_{11}\}$
        &
        $\w_{21}, \w_{22}, \w_{24}, \w_{27}, \w_{29}, \w_{30}, \w_{35}$
   \\  
	\hline
        $\w_{31}$
        &
		$\{\mathbf{ v}_6,\mathbf{ v}_5, \mathbf{ v}_3, \mathbf{ v}_8, 
		\mathbf{ v}_9, \mathbf{ v}_{10}, \mathbf{ v}_{11}\}$
        &
        $\w_{18}, \w_{19}, \w_{24}, \w_{27}, \w_{36}, \w_{40}, \w_{44}$
   \\ 
   \hline
       $\w_{34}$
        &
		$\{\mathbf{ v}_6,\mathbf{ v}_5, \mathbf{ v}_4, \mathbf{ v}_8, 
		\mathbf{ v}_9, \mathbf{ v}_{10}, \mathbf{ v}_{12}\}$
        & 
        $\w_{27}, \w_{35}, \w_{39}, \w_{43}, \w_{44}, \w_{57}, \w_{64}$
	\\
      	\hline
        $\w_{35}$
        &
		$\{\mathbf{ v}_6,\mathbf{ v}_5, \mathbf{ v}_4, \mathbf{ v}_7, 
		\mathbf{ v}_8, \mathbf{ v}_9, \mathbf{ v}_{12}\}$
        & 
        $\w_{28}, \w_{34}, \w_{39}, \w_{43}, \w_{54}, \w_{55}, \w_{56}$
        \\
        \hline
        $\w_{44}$
        &
		$\{\mathbf{ v}_6,\mathbf{ v}_5, \mathbf{ v}_3, \mathbf{ v}_8, 
		\mathbf{ v}_9, \mathbf{ v}_{10}, \mathbf{ v}_{12}\}$
        & 
        $\w_{31}, \w_{34}, \w_{52}, \w_{53}, \w_{56}, \w_{60}, \w_{61}$
	    \\
       \hline
        $\w_{53}$
        &
		$\{\mathbf{ v}_6,\mathbf{ v}_2, \mathbf{ v}_3, \mathbf{ v}_8, 
		\mathbf{ v}_9, \mathbf{ v}_{10}, \mathbf{ v}_{12}\}$
        & 
        $\w_{19}, \w_{20}, \w_{26}, \w_{44}, \w_{47}, \w_{50}, \w_{57}$
       \\ 
 	\hline
        $\w_{56}$
        &
		$\{\mathbf{ v}_6,\mathbf{ v}_5, \mathbf{ v}_3, \mathbf{ v}_7, 
		\mathbf{ v}_8, \mathbf{ v}_9, \mathbf{ v}_{12}\}$
        & 
        $\w_{24}, \w_{25}, \w_{26}, \w_{35}, \w_{44}, \w_{60}, \w_{61}$
      \\
	\hline
\end{tabular}
\end{center}
\caption{}
\label{tab:neighbor}
\end{table}
The other sets of neighbors unlisted in Table~\ref{tab:neighbor} can
be obtained by the symmetry of $V$.

\subsection{Computation data in Section \ref{sec:localization}}

In this subsection, we list all of the data, which are needed to
compute $f$ and $\mathcal{F}_{\mathrm{Td}^2}$.  First, we compute that
$f\equiv0$ by using its original localization formula.  Since $f$ is a
linear function in the variables $a_i,\!b_j,c$ ($1\le \!i, j\le \!3$) and
is symmetric among $\{a_i\}_{1\le i \le 3}$ and among $\{b_j\}_{1\le
  j\le 3}$ due to the symmetry of $V$, we can write
$$
	f(X)
	=A\sum_{i=1}^3 a_i + B\sum_{j=1}^3 b_j +Cc
$$
for some real numbers $A,B$ and $C$.  To show the vanishing of $f$, it
suffices to show that $f$ vanishes with respect to at least three
cases, for example,
$$
	(a_1, a_2, a_3, b_1, b_2, b_3, c)
	=
	\left\{\begin{array}{l}
	(-1,-1,-1, 1,2,3,1) 
	\\
	(1,1,1,1,2,3,4) 
	\\
	(-2,-2,-2, 1,2,3,1).
	\end{array}\right.
$$

The data of the first case is given in Table \ref{tab:Td1case1}.
\begin{table}
{\small
\begin{center}
\begin{tabular}{|c|c|l|l|l|l}
	\hline
		no. & fixed pt & $L(X)$ & $\det L(X)$ & $c_1(L(X))$ 
	\\
	\hline
			1-1&
		$(+++, \mathbf{ p}_1,\pm)$ &
		$(
			-1, -1, -1, 
			1, 2, -1, \pm3
		)$
		&
		$\pm6$
		& $(2, -4)$ 
	\\
	\hline
			1-2&
		$(-++, \mathbf{ p}_1,\pm)$ &
		$(
			1, -1, -1, 
			1, 2, -1, \pm4
		)$
		&
		$\mp8$
		& $(5,-3)$ 
	\\
	\hline
			1-3&
		$(+-+, \mathbf{ p}_1,\pm)$ &
		$(
			-1, 1, -1, 
			1, 2, -1, \pm4
		)$
		&
		$\mp8$
		& $(5,-3)$ 
	\\
	\hline
			1-4&
		$(++-, \mathbf{ p}_1,\pm)$ &
		$(
			-1, -1, 1, 
			1, 2, -1, \pm4
		)$
		&
		$\mp8$
		& $(5,-3)$ 
	\\
	\hline
			1-5&
		$(+--, \mathbf{ p}_1,\pm)$ &
		$(
			-1, 1, 1, 
			1, 2, -1, \pm5
		)$
		&
		$\pm10$
		& $(8,-2)$ 
	\\
	\hline
			1-6&
		$(-+-, \mathbf{ p}_1,\pm)$ &
		$(
			1, -1, 1, 
			1, 2, -1, \pm5
		)$
		&
		$\pm10$
		& $(8,-2)$ 
	\\
	\hline
			1-7&
		$(--+, \mathbf{ p}_1,\pm)$ &
		$(
			1, 1, -1, 
			1, 2, -1, \pm5
		)$
		&
		$\pm10$
		& $(8,-2)$ 
	\\
	\hline
			1-8&
		$(---, \mathbf{ p}_1,\pm)$ &
		$(
			1, 1, 1, 
			1, 2, -1, \pm6
		)$
		&
		$\mp12$
		& $(11,-1)$ 
	\\
	\hline
			2-1&
		$(+++, \mathbf{ p}_2,\pm)$ &
		$(
			-1, -1, -1, 
			-1, 1, -2, \pm5
		)$
		&
		$\mp10$
		& $(0,-10)$ 
	\\
	\hline
			2-2&
		$(-++, \mathbf{ p}_2,\pm)$ &
		$(
			1, -1, -1, 
			-1, 1, -2, \pm6
		)$
		&
		$\pm12$
		& $(3,-9)$ 
	\\
	\hline
			2-3&
		$(+-+, \mathbf{ p}_2,\pm)$ &
		$(
			-1, 1, -1, 
			-1, 1, -2, \pm6
		)$
		&
		$\pm12$
		& $(3,-9)$ 
	\\
	\hline
			2-4&
		$(++-, \mathbf{ p}_2,\pm)$ &
		$(
			-1, -1, 1, 
			-1, 1, -2, \pm6
		)$
		&
		$\pm12$
		& $(3,-9)$ 
	\\
	\hline
			2-5&
		$(+--, \mathbf{ p}_2,\pm)$ &
		$(
			-1, 1, 1, 
			-1, 1, -2, \pm7
		)$
		&
		$\mp14$
		& $(6,-8)$ 
	\\
	\hline
			2-6&
		$(-+-, \mathbf{ p}_2,\pm)$ &
		$(
			1, -1, 1, 
			-1, 1, -2, \pm7
		)$
		&
		$\mp14$
		& $(6,-8)$ 
	\\
	\hline
			2-7&
		$(--+, \mathbf{ p}_2,\pm)$ &
		$(
			1, 1, -1, 
			-1, 1, -2, \pm7
		)$
		&
		$\mp14$
		& $(6,-8)$ 
	\\
	\hline
			2-8&
		$(---, \mathbf{ p}_2,\pm)$ &
		$(
			1, 1, 1, 
			-1, 1, -2, \pm8
		)$
		&	
		$\pm16$
		& $(9,-7)$ 
	\\
	\hline
			3-1&
		$(+++, \mathbf{ p}_3,\pm)$ &
		$(
			-1, -1, -1, 
			-2, -1, -3, \pm7
		)$
		&
		$\pm42$
		& $(-2,-16)$ 
	\\
	\hline
			3-2&
		$(-++, \mathbf{ p}_3,\pm)$ &
		$(
			1, -1, -1, 
			-2, -1, -3, \pm8
		)$
		&
		$\mp48$
		& $(1,-15)$ 
	\\
	\hline
			3-3&
		$(+-+, \mathbf{ p}_3,\pm)$ &
		$(
			-1, 1, -1, 
			-2, -1, -3, \pm8
		)$
		&
		$\mp48$
		& $(1,-15)$ 
	\\
	\hline
			3-4&
		$(++-, \mathbf{ p}_3,\pm)$ &
		$(
			-1, -1, 1, 
			-2, -1, -3, \pm9
		)$
		&
		$\mp48$
		& $(1,-15)$ 
	\\
	\hline
			3-5&
		$(+--, \mathbf{ p}_3,\pm)$ &
		$(
			-1, 1, 1, 
			-2, -1, -3, \pm9
		)$
		&
		$\pm54$
		& $(4,-14)$ 
	\\
	\hline
			3-6&
		$(-+-, \mathbf{ p}_3,\pm)$ &
		$(
			1, -1, 1, 
			-2, -1, -3, \pm9
		)$
		&
		$\pm54$
		& $(4,-14)$ 
	\\
	\hline
			3-7&
		$(--+, \mathbf{ p}_3,\pm)$ &
		$(
			1, 1, -1, 
			-2, -1, -3, \pm9
		)$
		&
		$\mp54$
		& $(4,-14)$ 
	\\
	\hline
			3-8&
		$(---, \mathbf{ p}_3,\pm)$ &
		$(
			1, 1, 1, 
			-2, -1, -3, \pm10
		)$
		&	
		$\mp60$
		& $(7, -13)$ 
	\\
	\hline
			4-1&
		$(+++, \mathbf{ p}_4,\pm)$ &
		$(
			-1, -1, -1, 
			1, 2, 3, \pm 1
		)$
		&
		$\mp6$
		& $(4, 2)$ 
	\\
	\hline
			4-2&
		$(-++, \mathbf{ p}_4,\pm)$ &
		$(
			1, -1, -1, 
			1, 2, 3, \pm2
		)$
		&
		$\pm12$
		& $(7,3)$ 
	\\
	\hline
			4-3&
		$(+-+, \mathbf{ p}_4,\pm)$ &
		$(
			-1, 1, -1, 
			1, 2, 3, \pm2
		)$
		&
		$\pm12$
		& $(7,3)$ 
	\\
	\hline
			4-4&
		$(++-, \mathbf{ p}_4,\pm)$ &
		$(
			-1, -1, 1, 
			1, 2, 3, \pm3
		)$
		&
		$\pm12$
		& $(7,3)$
	\\
	\hline
			4-5&
		$(+--, \mathbf{ p}_4,\pm)$ &
		$(
			-1, 1, 1, 
			1, 2, 3, \pm3
		)$
		&
		$\mp18$
		& $(10,4)$ 
	\\
	\hline
			4-6&
		$(-+-, \mathbf{ p}_4,\pm)$ &
		$(
			1, -1, 1, 
			1, 2, 3, \pm3
		)$
		&
		$\mp18$
		& $(10,4)$ 
	\\
	\hline
			4-7&
		$(--+, \mathbf{ p}_4,\pm)$ &
		$(
			1, 1, -1, 
			1, 2, 3, \pm3
		)$
		&
		$\mp18$
		& $(10,4)$ 
	\\
	\hline
			4-8&
		$(---, \mathbf{ p}_4,\pm)$ &
		$(
			1, 1, 1, 
			1, 2, 3, \pm4
		)$
		&	
		$\pm24$
		& $(13,5)$ 
	\\
	\hline
\end{tabular}
\end{center}
\caption{}
\label{tab:Td1case1}
}
\end{table}
In the columns of $c_1(L(X))$ in Table \ref{tab:Td1case1}, the first
element corresponds to $(+)$-case and the other to $(-)$-case.  Our
computation is divided into the four parts of 
Table~\ref{tab:Td1case1}, which are labeled $\{(\mbox{1-}i)\}_{1\le i \le
  8}$, $\{(\mbox{2-}i)\}_{1\le i \le 8}$, $\{(\mbox{3-}i)\}_{1\le i
  \le 8}$ and $\{(\mbox{4-}i)\}_{1\le i \le 8}$.  
The sum among
$\{(\mbox{1-}i)\}_{1\le i \le 8}$ is given by
$$
	\frac{2^8}{6}-\frac{(-4)^8}{6}
	+3\bigg(\!-\frac{5^8}{8}+\frac{(-3)^8}{8}\bigg)
	+3\bigg(\frac{8^8}{10}-\frac{(-2)^8}{10}\bigg)
	-\frac{11^8}{12}+\frac{(-1)^8}{12}
	\!=\!-12985056.
$$
The sum among $\{(\mbox{2-}i)\}_{1\le i \le 8}$ is given by
$$
	-\frac{0^8}{10}+\frac{(-10)^8}{10}
	+3\bigg(\frac{3^8}{12}-\frac{(-9)^8}{12}\bigg)
	+3\bigg(-\frac{6^8}{14}+\frac{(-8)^8}{14}\bigg)
	+\frac{9^8}{16}-\frac{(-7)^8}{16}
	=4805280.
$$
The sum among $\{(\mbox{3-}i)\}_{1\le i \le 8}$ is given by
\begin{multline*}
	\frac{(-2)^8}{42}-\frac{(-16)^8}{42}
	+3\bigg(-\frac{1^8}{48}+\frac{(-15)^8}{48}\bigg)
	+3\bigg(\frac{4^8}{54}-\frac{(-14)^8}{54}\bigg)
	-\frac{7^8}{60}+\frac{(-13)^8}{60}
	=-10565664.
\end{multline*}
The sum among $\{(\mbox{4-}i)\}_{1\le i \le 8}$ is given by
$$
	-\frac{4^8}{6}+\frac{2^8}{6}
	+3\bigg(\frac{7^8}{12}-\frac{3^8}{12}\bigg)
	+3\bigg(-\frac{10^8}{18}+\frac{4^8}{18}\bigg)
	+\frac{13^8}{24}-\frac{5^8}{24}
	=18745440.
$$
Then, the total sum is equal to zero.

The data of the second case is given in Table \ref{tab:Td2}.  In this
case, Table \ref{tab:Td2} coincides with Table~\ref{tab:Td1case1} up
to order.  For example, the row (1-1) and (1-2) in Table~\ref{tab:Td2}
coincide with the row (1-8) and (1-7) in Table~\ref{tab:Td1case1}
respectively.  Hence, $f$ in this case also vanishes.

The data of the third case is given in Table \ref{tab:Td1case3}.

\begin{table}
{\small
\begin{center}
\begin{tabular}{|c|c|l|l|l|l}
	\hline
		no. & fixed pt & $L(X)$ & $\det L(X)$ & $c_1(L(X))$ 
	\\
	\hline
			1-1&
		$(+++, \mathbf{ p}_1,\pm)$ &
		$(
			-2, -2, -2, 
			1, 2, -1, \pm3
		)$
		&
		$\pm48$
		& $(-1, -7)$ 
	\\
	\hline
			1-2&
		$(-++, \mathbf{ p}_1,\pm)$ &
		$(
			2, -2, -2, 
			1, 2, -1, \pm5
		)$
		&
		$\mp80$
		& $(5,-5)$ 
	\\
	\hline
			1-3&
		$(+-+, \mathbf{ p}_1,\pm)$ &
		$(
			-2, 2, -2, 
			1, 2, -1, \pm5
		)$
		&
		$\mp80$
		& $(5,-5)$ 
	\\
	\hline
			1-4&
		$(++-, \mathbf{ p}_1,\pm)$ &
		$(
			-2, -2, 2, 
			1, 2, -1, \pm5
		)$
		&
		$\mp80$
		& $(5,-5)$ 
	\\
	\hline
			1-5&
		$(+--, \mathbf{ p}_1,\pm)$ &
		$(
			-2, 2, 2, 
			1, 2, -1, \pm7
		)$
		&
		$\pm112$
		& $(11,-3)$ 
	\\
	\hline
			1-6&
		$(-+-, \mathbf{ p}_1,\pm)$ &
		$(
			2, -2, 2, 
			1, 2, -1, \pm7
		)$
		&
		$\pm112$
		& $(11,-3)$ 
	\\
	\hline
			1-7&
		$(--+, \mathbf{ p}_1,\pm)$ &
		$(
			2, 2, -2, 
			1, 2, -1, \pm7
		)$
		&
		$\pm112$
		& $(11,-3)$ 
	\\
	\hline
			1-8&
		$(---, \mathbf{ p}_1,\pm)$ &
		$(
			2, 2, 2, 
			1, 2, -1, \pm9
		)$
		&
		$\mp144$
		& $(17,-1)$ 
	\\
	\hline
			2-1&
		$(+++, \mathbf{ p}_2,\pm)$ &
		$(
			-2, -2, -2, 
			-1, 1, -2, \pm5
		)$
		&
		$\mp80$
		& $(-3,-13)$ 
	\\
	\hline
			2-2&
		$(-++, \mathbf{ p}_2,\pm)$ &
		$(
			2, -2, -2, 
			-1, 1, -2, \pm7
		)$
		&
		$\pm112$
		& $(3,-11)$ 
	\\
	\hline
			2-3&
		$(+-+, \mathbf{ p}_2,\pm)$ &
		$(
			-2, 2, -2, 
			-1, 1, -2, \pm7
		)$
		&
		$\pm112$
		& $(3,-11)$ 
	\\
	\hline
			2-4&
		$(++-, \mathbf{ p}_2,\pm)$ &
		$(
			-2, -2, 2, 
			-1, 1, -2, \pm7
		)$
		&
		$\pm112$
		& $(3,-11)$ 
	\\
	\hline
			2-5&
		$(+--, \mathbf{ p}_2,\pm)$ &
		$(
			-2, 2, 2, 
			-1, 1, -2, \pm9
		)$
		&
		$\mp144$
		& $(9,-9)$ 
	\\
	\hline
			2-6&
		$(-+-, \mathbf{ p}_2,\pm)$ &
		$(
			2, -2, 2, 
			-1, 1, -2, \pm9
		)$
		&
		$\mp144$
		& $(9,-9)$ 
	\\
	\hline
			2-7&
		$(--+, \mathbf{ p}_2,\pm)$ &
		$(
			2, 2, -2, 
			-1, 1, -2, \pm9
		)$
		&
		$\mp144$
		& $(9,-9)$ 
	\\
	\hline
			2-8&
		$(---, \mathbf{ p}_2,\pm)$ &
		$(
			2, 2, 2, 
			-1, 1, -2, \pm11
		)$
		&
		$\pm176$
		& $(15,-7)$ 
	\\
	\hline
			3-1&
		$(+++, \mathbf{ p}_3,\pm)$ &
		$(
			-2, -2, -2, 
			-2, -1, -3, \pm7
		)$
		&
		$\pm336$
		& $(-5,-19)$ 
	\\
	\hline
			3-2&
		$(-++, \mathbf{ p}_3,\pm)$ &
		$(
			2, -2, -2, 
			-2, -1, -3, \pm9
		)$
		&
		$\mp432$
		& $(1,-17)$ 
	\\
	\hline
			3-3&
		$(+-+, \mathbf{ p}_3,\pm)$ &
		$(
			-2, 2, -2, 
			-2, -1, -3, \pm9
		)$
		&
		$\mp432$
		& $(1,-17)$ 
	\\
	\hline
			3-4&
		$(++-, \mathbf{ p}_3,\pm)$ &
		$(
			-2, -2, 2, 
			-2, -1, -3, \pm9
		)$
		&
		$\mp432$
		& $(1,-17)$ 
	\\
	\hline
			3-5&
		$(+--, \mathbf{ p}_3,\pm)$ &
		$(
			-2, 2, 2, 
			-2, -1, -3, \pm11
		)$
		&
		$\pm528$
		& $(7,-15)$ 
	\\
	\hline
			3-6&
		$(-+-, \mathbf{ p}_3,\pm)$ &
		$(
			2, -2, 2, 
			-2, -1, -3, \pm11
		)$
		&
		$\pm528$
		& $(7,-15)$ 
	\\
	\hline
			3-7&
		$(--+, \mathbf{ p}_3,\pm)$ &
		$(
			2, 2, -2, 
			-2, -1, -3, \pm11
		)$
		&
		$\pm528$
		& $(7,-15)$ 
	\\
	\hline
			3-8&
		$(---, \mathbf{ p}_3,\pm)$ &
		$(
			2, 2, 2, 
			-2, -1, -3, \pm13
		)$
		&	
		$\mp624$
		& $(13, -13)$ 
	\\
	\hline
			4-1&
		$(+++, \mathbf{ p}_4,\pm)$ &
		$(
			-2, -2, -2, 
			1, 2, 3, \pm 1
		)$
		&
		$\mp48$
		& $(1, -1)$ 
	\\
	\hline
			4-2&
		$(-++, \mathbf{ p}_4,\pm)$ &
		$(
			2, -2, -2, 
			1, 2, 3, \pm3
		)$
		&
		$\pm144$
		& $(7,1)$ 
	\\
	\hline
			4-3&
		$(+-+, \mathbf{ p}_4,\pm)$ &
		$(
			-2, 2, -2, 
			1, 2, 3, \pm3
		)$
		&
		$\pm144$
		& $(7,1)$ 
	\\
	\hline
			4-4&
		$(++-, \mathbf{ p}_4,\pm)$ &
		$(
			-2, -2, 2, 
			1, 2, 3, \pm3
		)$
		&
		$\pm144$
		& $(7,1)$ 
	\\
	\hline
			4-5&
		$(+--, \mathbf{ p}_4,\pm)$ &
		$(
			-2, 2, 2, 
			1, 2, 3, \pm5
		)$
		&
		$\mp240$
		& $(13,3)$ 
	\\
	\hline
			4-6&
		$(-+-, \mathbf{ p}_4,\pm)$ &
		$(
			2, -2, 2, 
			1, 2, 3, \pm5
		)$
		&
		$\mp240$
		& $(13,3)$ 
	\\
	\hline
			4-7&
		$(--+, \mathbf{ p}_4,\pm)$ &
		$(
			2, 2, -2, 
			1, 2, 3, \pm5
		)$
		&
		$\mp240$
		& $(13,3)$ 
	\\
	\hline
			4-8&
		$(---, \mathbf{ p}_4,\pm)$ &
		$(
			2, 2, 2, 
			1, 2, 3, \pm7
		)$
		&	
		$\pm336$
		& $(19,5)$ 
	\\
	\hline
\end{tabular}
\end{center}
\caption{}
\label{tab:Td1case3}
}
\end{table}
Then, the sum among $\{(\mbox{1-}i)\}_{1\le i \le 8}$ is given by
\begin{multline*}
	\frac{(-1)^8}{48}-\frac{(-7)^8}{48}
	+3\bigg(-\frac{5^8}{80}+\frac{(-5)^8}{80}\bigg)
	+3\bigg(\frac{11^8}{112}-\frac{(-3)^8}{112}\bigg)
	-\frac{17^8}{114}+\frac{(-1)^8}{114}
	=-42821280.
\end{multline*}
The sum among $\{(\mbox{2-}i)\}_{1\le i \le 8}$ is given by
\begin{multline*}
	-\frac{(-3)^8}{80}+\frac{(-13)^8}{80}
	+3\bigg(\frac{3^8}{112}-\frac{(-11)^8}{112}\bigg)
	+3\bigg(-\frac{9^8}{144}+\frac{(-9)^8}{144}\bigg)
	+\frac{15^8}{176}-\frac{(-7)^8}{176}
	=18984096.
\end{multline*}
The sum among $\{(\mbox{3-}i)\}_{1\le i \le 8}$ is given by
\begin{multline*}
	\frac{(-5)^8}{336}-\frac{(-19)^8}{336}
	+3\bigg(-\frac{1^8}{432}+\frac{(-17)^8}{432}\bigg)
	+3\bigg(\frac{7^8}{528}-\frac{(-15)^8}{528}\bigg)
	-\frac{13^8}{624}+\frac{(-13)^8}{624}
	=-16631520.
\end{multline*}
The sum among $\{(\mbox{4-}i)\}_{1\le i \le 8}$ is given by
$$
	-\frac{1^8}{48}+\frac{(-1)^8}{48}
	+3\bigg(\frac{7^8}{144}-\frac{1^8}{144}\bigg)
	+3\bigg(-\frac{13^8}{240}+\frac{3^8}{20}\bigg)
	+\frac{19^8}{336}-\frac{5^8}{336}
	=40468704.
$$
Then, the total sum is equal to zero.

Next, we compute $\mathcal{F}_{\mathrm{Td}^2}(X)$, where $X$ is the
holomorphic vector field associated with $\sigma_t$ for when $(a_1,
a_2, a_3, b_1, b_2, b_3, c)=(1,1,1,1,2,3,4)$.

\begin{table}
\begin{center}
\begin{tabular}{|c|c|l|l|l|l|}
	\hline
	no. & fixed pt & $L(X)$ & $\det L(X)$ & $c_1(L(X))$ & $c_2(L(X))$
	\\
	\hline
			1-1&
		$(+++, \mathbf{ p}_1,\pm)$ &
		$(
			1, 1, 1, 
			1, 2, -1, \pm6
		)$
		&
		$\mp12$
		& $(11, -1)$ & $(38, -22)$
	\\
	\hline
			1-2&
		$(-++, \mathbf{ p}_1,\pm)$ &
		$(
			-1, 1, 1, 
			1, 2, -1, \pm5
		)$
		&
		$\pm10$
		& $(8,-2)$ & $(15, -15)$
	\\
	\hline
			1-3&
		$(+-+, \mathbf{ p}_1,\pm)$ &
		$(
			1, -1, 1, 
			1, 2, -1, \pm5
		)$
		&
		$\pm10$
		& $(8,-2)$ & $(15,-15)$
	\\
	\hline
			1-4&
		$(++-, \mathbf{ p}_1,\pm)$ &
		$(
			1, 1, -1, 
			1, 2, -1, \pm5
		)$
		&
		$\pm10$
		& $(8,-2)$ & $(15,-15)$
	\\
	\hline
			1-5&
		$(+--, \mathbf{ p}_1,\pm)$ &
		$(
			1, -1, -1, 
			1, 2, -1, \pm4
		)$
		&
		$\mp8$
		& $(5,-3)$ & $(0, -8)$
	\\
	\hline
			1-6&
		$(-+-, \mathbf{ p}_1,\pm)$ &
		$(
			-1, 1, -1, 
			1, 2, -1, \pm4
		)$
		&
		$\mp8$
		& $(5,-3)$ & $(0, -8)$
	\\
	\hline
			1-7&
		$(--+, \mathbf{ p}_1,\pm)$ &
		$(
			-1, -1, 1, 
			1, 2, -1, \pm4
		)$
		&
		$\mp8$
		& $(5,-3)$ & $(0, -8)$
	\\
	\hline
			1-8&
		$(---, \mathbf{ p}_1,\pm)$ &
		$(
			-1, -1, -1, 
			1, 2, -1, \pm3
		)$
		&
		$\pm6$
		& $(2,-4)$ & $(-7, -1)$
	\\
	\hline
			2-1&
		$(+++, \mathbf{ p}_2,\pm)$ &
		$(
			1, 1, 1, 
			-1, 1, -2, \pm8
		)$
		&
		$\pm16$
		& $(9,-7)$ & $(4, -12)$
	\\
	\hline
			2-2&
		$(-++, \mathbf{ p}_2,\pm)$ &
		$(
			-1, 1, 1, 
			-1, 1, -2, \pm7
		)$
		&
		$\mp14$
		& $(6,-8)$ & $(-11, 3)$
	\\
	\hline
			2-3&
		$(+-+, \mathbf{ p}_2,\pm)$ &
		$(
			1, -1, 1, 
			-1, 1, -2, \pm7
		)$
		&
		$\mp14$
		& $(6,-8)$ & $(-11,3)$
	\\
	\hline
			2-4&
		$(++-, \mathbf{ p}_2,\pm)$ &
		$(
			1, 1, -1, 
			-1, 1, -2, \pm7
		)$
		&
		$\mp14$
		& $(6,-8)$ & $(-11,3)$
	\\
	\hline
			2-5&
		$(+--, \mathbf{ p}_2,\pm)$ &
		$(
			1, -1, -1, 
			-1, 1, -2, \pm6
		)$
		&
		$\pm12$
		& $(3,-9)$ & $(-18, 18)$
	\\
	\hline
			2-6&
		$(-+-, \mathbf{ p}_2,\pm)$ &
		$(
			-1, 1, -1, 
			-1, 1, -2, \pm6
		)$
		&
		$\pm12$
		& $(3,-9)$ & $(-18,18)$
	\\
	\hline
			2-7&
		$(--+, \mathbf{ p}_2,\pm)$ &
		$(
			-1, -1, 1, 
			-1, 1, -2, \pm6
		)$
		&
		$\pm12$
		& $(3,-9)$ & $(-18,18)$
	\\
	\hline
			2-8&
		$(---, \mathbf{ p}_2,\pm)$ &
		$(
			-1, -1, -1, 
			-1, 1, -2, \pm5
		)$
		&	
		$\mp10$
		& $(0,-10)$ & $(-17,33)$
	\\
	\hline
			3-1&
		$(+++, \mathbf{ p}_3,\pm)$ &
		$(
			1, 1, 1, 
			-2, -1, -3, \pm10
		)$
		&
		$\mp60$
		& $(7,-13)$ & $(-34,26)$
	\\
	\hline
			3-2&
		$(-++, \mathbf{ p}_3,\pm)$ &
		$(
			-1, 1, 1, 
			-2, -1, -3, \pm9
		)$
		&
		$\pm54$
		& $(4,-14)$ & $(-41,49)$
	\\
	\hline
			3-3&
		$(+-+, \mathbf{ p}_3,\pm)$ &
		$(
			1, -1, 1, 
			-2, -1, -3, \pm9
		)$
		&
		$\pm54$
		& $(4,-14)$ & $(-41,49)$
	\\
	\hline
			3-4&
		$(++-, \mathbf{ p}_3,\pm)$ &
		$(
			1, 1, -1, 
			-2, -1, -3, \pm9
		)$
		&
		$\pm54$
		& $(4,-14)$ & $(-41,49)$
	\\
	\hline
			3-5&
		$(+--, \mathbf{ p}_3,\pm)$ &
		$(
			1, -1, -1, 
			-2, -1, -3, \pm8
		)$
		&
		$\mp48$
		& $(1,-15)$ & $(-40, 72)$
	\\
	\hline
			3-6&
		$(-+-, \mathbf{ p}_3,\pm)$ &
		$(
			-1, 1, -1, 
			-2, -1, -3, \pm8
		)$
		&
		$\mp48$
		& $(1,-15)$ & $(-40,72)$
	\\
	\hline
			3-7&
		$(--+, \mathbf{ p}_3,\pm)$ &
		$(
			-1, -1, 1, 
			-2, -1, -3, \pm8
		)$
		&
		$\mp48$
		& $(1,-15)$ & $(-40,72)$
	\\
	\hline
			3-8&
		$(---, \mathbf{ p}_3,\pm)$ &
		$(
			-1, -1, -1, 
			-2, -1, -3, \pm7
		)$
		&	
		$\pm42$
		& $(-2, -16)$ & $(-31,95)$
	\\
	\hline
			4-1&
		$(+++, \mathbf{ p}_4,\pm)$ &
		$(
			1, 1, 1, 
			1, 2, 3, \pm 4
		)$
		&
		$\pm24$
		& $(13, 5)$ & $(68, -4)$
	\\
	\hline
			4-2&
		$(-++, \mathbf{ p}_4,\pm)$ &
		$(
			-1, 1, 1, 
			1, 2, 3, \pm3
		)$
		&
		$\mp18$
		& $(10,4)$ & $(37,-5)$
	\\
	\hline
			4-3&
		$(+-+, \mathbf{ p}_4,\pm)$ &
		$(
			1, -1, 1, 
			1, 2, 3, \pm3
		)$
		&
		$\mp18$
		& $(10,4)$ & $(37,-5)$
	\\
	\hline
			4-4&
		$(++-, \mathbf{ p}_4,\pm)$ &
		$(
			1, 1, -1, 
			1, 2, 3, \pm3
		)$
		&
		$\mp18$
		& $(10,4)$ & $(37,-5)$
	\\
	\hline
			4-5&
		$(+--, \mathbf{ p}_4,\pm)$ &
		$(
			1, -1, -1, 
			1, 2, 3, \pm2
		)$
		&
		$\pm12$
		& $(7,3)$ & $(14,-6)$
	\\
	\hline
			4-6&
		$(-+-, \mathbf{ p}_4,\pm)$ &
		$(
			-1, 1, -1, 
			1, 2, 3, \pm2
		)$
		&
		$\pm12$
		& $(7,3)$ & $(14,-6)$
	\\
	\hline
			4-7&
		$(--+, \mathbf{ p}_4,\pm)$ &
		$(
			-1, -1, 1, 
			1, 2, 3, \pm2
		)$
		&
		$\pm12$
		& $(7,3)$ & $(14,-6)$
	\\
	\hline
			4-8&
		$(---, \mathbf{ p}_4,\pm)$ &
		$(
			-1, -1, -1, 
			1, 2, 3, \pm1
		)$
		&	
		$\mp6$
		& $(4,2)$ & $(-1,-7)$
	\\
	\hline
\end{tabular}
\end{center}
\caption{}
\label{tab:Td2}
\end{table}

The data is given in Table \ref{tab:Td2}.
Since $\mathcal{F}_{\mathrm{Td}^1}$ vanishes, it is sufficient to check that 
\begin{equation}\label{eq:check}
	\sum_{\mathbf{q}:\, \mbox{fixed pt}}
    \frac{(c_2c_1^6)(L(X)_{\mathbf{q}})}{\det(L(X)_{\mathbf{q}})}
\end{equation}
does not vanish.
Then, we calculate (\ref{eq:check}) separately as follows.
The sum among $\{(\mbox{1-}i)\}_{1\le i \le 8}$ is given by
$$
		-\frac{38\cdot11^6}{12}-\frac{22}{12}
		+3\biggl(\frac{15\cdot8^6}{10}+\frac{15\cdot2^6}{10}\biggr)
		-3\frac{8\cdot3^6}{8}
		-\frac{7\cdot2^6}{6}+\frac{4^6}{6}
		=
		-4431588,
$$
the sum among $\{(\mbox{2-}i)\}_{1\le i \le 8}$ is given by
$$
		\frac{4\cdot9^6}{16}+\frac{12\cdot7^6}{16}
		+3\biggl(\frac{11\cdot6^6}{14}+\frac{3\cdot8^6}{14}\biggr)
		-3\biggl(\frac{18\cdot3^6}{12}+\frac{18\cdot9^6}{12}\biggr)	
		+\frac{33\cdot10^6}{10}\\
		=1404828,
$$
the sum among $\{(\mbox{3-}i)\}_{1\le i \le 8}$ is given by
$$
		\frac{34\cdot7^6}{60}+\frac{26\cdot13^6}{60}
		-3\biggl(\frac{41\cdot4^6}{54}+\frac{49\cdot14^6}{54}\biggr)
		+3\biggl(\frac{40}{48}+\frac{72\cdot15^6}{48}\biggr)
		-\frac{31\cdot2^6}{42}-\frac{95\cdot16^6}{42}
		=
		-5038812,
$$
and the sum among $\{(\mbox{4-}i)\}_{1\le i \le 8}$ is given by
$$
		\frac{68\cdot 13^6}{24}+\frac{4\cdot5^6}{24}
		-3\biggl(\frac{37\cdot10^6}{18}+\frac{5\cdot4^6}{18}\biggr)
		+3\biggl(\frac{14\cdot7^6}{12}+\frac{6\cdot3^6}{12}\biggr)
		+\frac{4^6}{6}-\frac{7\cdot2^6}{6}
		=
		7921956.
$$
Therefore, we get the total sum
\begin{align}
	\nonumber
\sum_{\mathbf{q}:\, \mbox{fixed pt}}
\frac{(c_2c_1^6)(L(X)_{\mathbf{q}})}{\det(L(X)_{\mathbf{q}})}
	&=
	-4431588+1404828-5038812+7921956
	\\
	\nonumber
	&=
	-143616
	\\
	\label{eq:check_coefficient_2}
	&=
	13056\times(-11)\neq 0.
\end{align}
The last equality (\ref{eq:check_coefficient_2}) confirms the result
(\ref{eq:check_coefficient_1}) in Section~\ref{sec:localization}.


\end{document}